\documentclass{amsart}
\usepackage{amsfonts}
\usepackage{amsmath}

\newtheorem{theorem}{Theorem}
\theoremstyle{plain}
\newtheorem{acknowledgement}{Acknowledgement}

\newtheorem{condition}{Condition}

\newtheorem{definition}{Definition}
\newtheorem{example}{Example}

\newtheorem{lemma}{Lemma}

\newtheorem{proposition}{Proposition}
\newtheorem{remark}{Remark}

\numberwithin{equation}{section}
\input{tcilatex}
\input tcilatex

\begin{document}
\def\cprime{$'$}
\def\cprime{$'$}

\title[A (rough) pathwise approach to nonlinear SPDEs]{A (rough) pathwise
approach to a class of nonlinear stochastic partial differential equations}
\author{Michael Caruana, Peter K. Friz and Harald Oberhauser}
\address{MC\ is affiliated to King Fahd University of Petroleum and
Minerals. HO is affiliated to TU Berlin. PKF is corresponding author
(friz@math.tu-berlin.de) and affiliated to TU and WIAS Berlin. }

\begin{abstract}
We consider nonlinear parabolic evolution equations of the form $\partial
_{t}u=F\left( t,x,Du,D^{2}u\right) $, subject to noise of the form $H\left(
x,Du\right) \circ dB$ where $H$ is linear in $Du$ and $\circ dB$ denotes the
Stratonovich differential of a multidimensional Brownian motion. Motivated
by the essentially pathwise results of [Lions, P.-L. and Souganidis, P.E.;
Fully nonlinear stochastic partial differential equations. C. R. Acad. Sci.
Paris S\'{e}r. I Math. 326 (1998), no. 9] we propose the use of rough path
analysis [Lyons, T. J.; Differential equations driven by rough signals. Rev.
Mat. Iberoamericana 14 (1998), no. 2, 215--310] in this context. Although
the core arguments are entirely deterministic, a continuity theorem allows
for various probabilistic applications (limit theorems, support, large
deviations, ...).
\end{abstract}

\maketitle

\section{Introduction}

Let us recall some basic ideas of (second order) viscosity theory \cite%
{MR1118699UserGuide, MR2179357FS} and rough path theory \cite{lyons-qian-02,
MR2314753}.\ As for viscosity theory, consider a real-valued function $%
u=u\left( x\right) $ with $x\in \mathbb{R}^{n}$ and assume $u\in C^{2}$ is a
classical supersolution,%
\begin{equation*}
-G\left( x,u,Du,D^{2}u\right) \geq 0,
\end{equation*}%
where $G$ is a (continuous) function, \textit{degenerate elliptic} in the
sense that $G\left( x,u,p,A\right) \leq G\left( x,u,p,A+B\right) $ whenever $%
B\geq 0$ in the sense of symmetric matrices. The idea is to consider a
(smooth) test function $\varphi $ which touches $u$ from below at some point 
$\bar{x}$. Basic calculus implies that $Du\left( \bar{x}\right) =D\varphi
\left( \bar{x}\right) ,\,D^{2}u\left( \bar{x}\right) \geq D^{2}\varphi
\left( \bar{x}\right) $ and, from degenerate ellipticity,%
\begin{equation}
-G\left( \bar{x},\varphi ,D\varphi ,D^{2}\varphi \right) \,\geq 0.
\label{Gxbar}
\end{equation}%
This suggests to define a \textit{viscosity supersolution} (at the point $%
\bar{x}$) to $-G=0$ as a (lower semi-)continuous function $u$ with the
property that (\ref{Gxbar}) holds for any test function which touches $u$
from below at $\bar{x}$. Similarly, \textit{viscosity subsolutions} are
(upper semi-)continuous functions defined via test functions touching $u$
from above and by reversing inequality in (\ref{Gxbar}); \textit{viscosity
solutions} are both super- and subsolutions. Observe that this definition
covers (completely degenerate) first order equations as well as parabolic
equations, e.g. by considering $\partial _{t}-F=0$ on $\mathbb{R}^{+}\times 
\mathbb{R}^{n}$ where $F$ is degenerate elliptic. The resulting theory
(existence, uniqueness, stability, ...) is without doubt one of most
important recent developments in the field of partial differential
equations. As a typical result\footnote{$\mathrm{BUC}\left( \dots \right) $
denotes the space of bounded, uniformly continuous functions; $\mathrm{BC}%
\left( \dots \right) $ denotes the space of bounded, continuous functions.},
one has existence and uniqueness result in the class of bounded solutions to
the initial value problem $\left( \partial _{t}-F\right) u=0,\,u\left(
0,\cdot \right) =u_{0}\in \mathrm{BUC}\left( \mathbb{R}^{n}\right) $,
provided $F=F(t,x,Du,D^{2}u)$ is continuous, degenerate elliptic and
satisfies a (well-known) technical condition (see condition \ref{UG314}
below). In fact, uniqueness follows from a stronger property known as 
\textit{comparison:} assume $u$ (resp. $v$) is a subsolution (resp.
supersolution) and $u_{0}\leq v_{0}$; then $u\leq v$ on $[0,T)\times \mathbb{%
R}^{n}$. A key feature of viscosity theory is what workers in the field
simply call \textit{stability properties}. For instance, it is relatively
straight-forward to study $\left( \partial _{t}-F\right) u=0$ via a sequence
of approximate problems, say $\left( \partial _{t}-F^{n}\right) u^{n}=0$,
provided $F^{n}\rightarrow F$ locally uniformly and some apriori information
on the $u^{n}$ (e.g. locally uniform convergence, or locally uniform
boundedness\footnote{%
What we have in mind here is the \textit{Barles--Perthame method of
semi-relaxed limits}. We shall use this method in the proof of theorem \ref%
{main} and postpone precise references until then.}). Note the stark
contrast to the classical theory where one has to control the actual
deriviatives of $u^{n}$.

\bigskip

The idea of stability is also central to \textit{rough path theory}. Given a
collection $\left( V_{1},\dots ,V_{d}\right) $ of (sufficiently nice) vector
fields on $\mathbb{R}^{n}$ and $z\in C^{1}\left( \left[ 0,T\right] ,\mathbb{R%
}^{d}\right) $ one considers the (unique) solution $y$ to the ordinary
differential equation%
\begin{equation}
\dot{y}\left( t\right) =\sum_{i=1}^{d}V_{i}\left( y\right) \dot{z}^{i}\left(
t\right) ,\,\,\,y\left( 0\right) =y_{0}\in \mathbb{R}^{n}\text{.}
\label{ODEintro}
\end{equation}%
The question is, if the output signal $y$ depends in a stable way on the
driving signal $z$. The answer, of course, depends strongly on how to
measure distance between input signals. If one uses the supremum norm, so
that the distance between driving signals $z,\tilde{z}$ is given by $%
\left\vert z-\tilde{z}\right\vert _{\infty ;\left[ 0,T\right] }$, then the
solution will in general \textit{not }depend continuously on the input.

\begin{example}
\label{ODEnotContInInfity}\bigskip Take $n=1,d=2,\,V=\left(
V_{1},V_{2}\right) =\left( \sin \left( \cdot \right) ,\cos \left( \cdot
\right) \right) $ and $y_{0}=0$. Obviously,%
\begin{equation*}
z^{n}\left( t\right) =\left( \frac{1}{n}\cos \left( 2\pi n^{2}t\right) ,%
\frac{1}{n}\sin \left( 2\pi n^{2}t\right) \right)
\end{equation*}%
converges to $0$ in $\infty $-norm whereas the solutions to $\dot{y}%
^{n}=V\left( y^{n}\right) \dot{z}^{n},y_{0}^{n}=0,$ do not converge to zero
(the solution to the limiting equation $\dot{y}=0$).
\end{example}

If $\left\vert z-\tilde{z}\right\vert _{\infty ;\left[ 0,T\right] }$ is
replaced by the (much) stronger distance 
\begin{equation*}
\left\vert z-\tilde{z}\right\vert _{1\text{-var};\left[ 0,T\right]
}=\sup_{\left( t_{i}\right) \subset \left[ 0,T\right] }\sum \left\vert
z_{t_{i},t_{i+1}}-\tilde{z}_{t_{i},t_{i+1}}\right\vert ,
\end{equation*}%
it is elementary to see that now the solution map is continuous (in fact,
locally Lipschitz); however, this continuity does not lend itself to push
the meaning of (\ref{ODEintro}): the closure of $C^{1}$ (or smooth) paths in
variation is precisely $W^{1,1}$, the set of absolutely continuous paths
(and thus still far from a typical Brownian path). Lyons' theory of rough
paths exhibits an entire cascade of ($p$-variation or $1/p$-H\"{o}lder type
rough path) metrics, for each $p\in \lbrack 1,\infty )$, on path-space under
which such ODE solutions are continuous (and even locally Lipschitz)
functions of their driving signal. For instance, the "rough path" $p$%
-variation distance between two smooth $\mathbb{R}^{d}$-valued paths $z,%
\tilde{z}$ is given by%
\begin{equation*}
\max_{j=1,\dots ,\left[ p\right] }\left( \sup_{\left( t_{i}\right) \subset %
\left[ 0,T\right] }\sum \left\vert z_{t_{i},t_{i+1}}^{\left( j\right) }-%
\tilde{z}_{t_{i},t_{i+1}}^{\left( j\right) }\right\vert ^{p}\right) ^{1/p}
\end{equation*}%
where $z_{s,t}^{\left( j\right) }=\int dz_{r_{1}}\otimes \dots \otimes
dz_{r_{j}}$ with integration over the $j$-dimensional simplex $\left\{
s<r_{1}<\dots <r_{j}<t\right\} $. This allows to extend the very meaning of (%
\ref{ODEintro}), in a unique and continuous fashion, to driving signals
which live in the abstract completion of smooth $\mathbb{R}^{d}$-valued
paths (with respect to rough path $p$-variation or a similarly defined $1/p$%
-H\"{o}lder metric). The space of so-called $p$-rough paths\footnote{%
In the strict terminology of rough path theory: geometric $p$-rough paths.}
is precisely this abstract completion. In fact, this space can be realized
as genuine path space, 
\begin{equation*}
C^{0,p\text{-var}}\left( \left[ 0,T\right] ,G^{\left[ p\right] }\left( 
\mathbb{R}^{d}\right) \right) \text{ \ resp. }C^{0,1/p\text{-H\"{o}l}}\left( %
\left[ 0,T\right] ,G^{\left[ p\right] }\left( \mathbb{R}^{d}\right) \right)
\end{equation*}%
where $G^{\left[ p\right] }\left( \mathbb{R}^{d}\right) $ is the free step-$%
\left[ p\right] $ nilpotent group over $\mathbb{R}^{d}$, equipped with
Carnot--Caratheodory metric; realized as a subset of $1+\mathfrak{t}^{\left[
p\right] }\left( \mathbb{R}^{d}\right) $ where%
\begin{equation*}
\mathfrak{t}^{\left[ p\right] }\left( \mathbb{R}^{d}\right) =\mathbb{R}%
^{d}\oplus \left( \mathbb{R}^{d}\right) ^{\otimes 2}\oplus \dots \oplus
\left( \mathbb{R}^{d}\right) ^{\otimes \left[ p\right] }
\end{equation*}%
is the natural state space for (up to $\left[ p\right] $) iterated integrals
of a smooth $\mathbb{R}^{d}$-valued path. For instance, almost every
realization of $d$-dimensional Brownian motion $B$ \textit{enhanced with its
iterated stochastic integrals in the sense of\ Stratonovich, }i.e. the
matrix-valued process given by%
\begin{equation}
B^{\left( 2\right) }:=\left( \int_{0}^{\cdot }B^{i}\circ dB^{j}\right)
_{i,j\in \left\{ 1,\dots ,d\right\} }  \label{StratoII_intro}
\end{equation}%
yields a path $\mathbf{B}\left( \omega \right) $ in $G^{2}\left( \mathbb{R}%
^{d}\right) $ with finite $1/p$-H\"{o}lder (and hence finite $p$-variation)
regularity, for any $p>2$. ($\mathbf{B}$ is known as \textit{Brownian rough
path.}) We remark that $B^{\left( 2\right) }=\frac{1}{2}B\otimes B+A$ where $%
A:=\mathrm{Anti}\left( B^{\left( 2\right) }\right) $ is known as \textit{L%
\'{e}vy's stochastic area;} in other words $\mathbf{B}\left( \omega \right) $
is determined by $\left( B,A\right) $, i.e. Brownian motion \textit{enhanced
with L\'{e}vy's area.}

Turning to the main topic of this paper, we follow \cite{MR1647162,
MR1659958, MR1807189} in considering a real-valued function of time and
space $u=u\left( t,x\right) \in \mathrm{BC}\left( [0,T)\times \mathbb{R}%
^{n}\right) $ which solves the nonlinear partial differential equation 
\begin{eqnarray}
du &=&F\left( t,x,Du,D^{2}u\right) dt+\sum_{i=1}^{d}H_{i}\left( x,Du\right)
dz^{i}  \label{PDE} \\
&\equiv &F\left( t,x,Du,D^{2}u\right) dt+H\left( x,Du\right) dz  \notag
\end{eqnarray}%
in viscosity sense. When $z:\left[ 0,T\right] \rightarrow \mathbb{R}^{d}$ is 
$C^{1}$ then, subject to suitable conditions on $F$ and $H$, this falls in
the standard setting of viscosity theory as discussed above. This can be
pushed further to $z\in W^{1,1}$ (see e.g. \cite[Remark 4]{MR1647162} and
the references given there) but the case when $z=z\left( t\right) $ has only
"Brownian" regularity (just below $1/2$-H\"{o}lder, say) falls dramatically
outside the scope of the standard theory. The reader can find a variety of
examples (drawing from fields as diverse as stochastic control theory,
pathwise stochastic control, interest rate theory, front propagation and
phase transition in random media, ...) in the articles \cite{MR1659958,
MR1959710} justifying the need of a theory of (non-linear)\ \textit{%
stochastic partial differential equations} (SPDEs) in which $z$ in\ (\ref%
{PDE}) is taken as a Brownian motion\footnote{%
... in which case (\ref{PDE}) is understood in Stratonovich form.}. In the
same series of articles a satisfactory theory is established for the case of
non-linear Hamiltonian with no spatial dependence, i.e. $H=H\left( Du\right) 
$. The contribution of this article is to deal with non-linear $F$ and $%
H=H\left( x,Du\right) $, linear in $Du$, although we suspect that the
marriage of rough path and viscosity methodology will also prove useful in
further investigations on \textit{fully nonlinear} (i.e. both $F$ and $H$)
stochastic partial differential equations\footnote{%
The use of rough path analysis in the context of nonlinear SPDEs was
verbally conjectured by P.L. Lions in his 2003 Courant lecture.}. To fix
ideas, we give the following example, suggested in \cite{MR1659958} and
carefully worked out in \cite{MR1920103, MR2285722}.

\begin{example}[Pathwise stochastic control]
\label{ExStochControl}Consider%
\begin{equation*}
dX=b\left( X;\alpha \right) dt+W\left( X;\alpha \right) \circ d\tilde{B}%
+V\left( X\right) \circ dB,
\end{equation*}%
where $b,W,V$ are (collections of) sufficiently nice vector fields (with $%
b,W $ dependent on a suitable control $\alpha =\alpha \left( t\right) \in 
\mathcal{A}$, applied at time $t$) and $\tilde{B},B$ are multi-dimensional
(independent)\ Brownian motions. Define\footnote{%
Remark that any optimal control $\alpha \left( \cdot \right) $ here will
depend on knowledge of the entire path of $B$. Such anticipative control
problems and their link to classical stochastic control problems were
discussed early on by Davis and Burnstein \cite{MR1275134}.}%
\begin{equation*}
v\left( x,t;B\right) =\inf_{\alpha \in \mathcal{A}}\mathbb{E}\left[ \left.
\left( g\left( X_{T}^{x,t}\right) +\int_{t}^{T}f\left( X_{s}^{x,t},\alpha
_{s}\right) ds\right) \right\vert B\right]
\end{equation*}%
where $X^{x,t}$ denotes the solution process to the above SDE started at $%
X\left( t\right) =x$. Then, at least by a formal computation,%
\begin{equation*}
dv+\inf_{\alpha \in \mathcal{A}}\left[ b\left( x,\alpha \right)
Dv+L_{a}v+f\left( x,\alpha \right) \right] dt+Dv\cdot V\left( x\right) \circ
dB=0
\end{equation*}%
with terminal data $v\left( \cdot ,T\right) \equiv g$, and $L_{\alpha }=\sum
W_{i}^{2}$ in H\"{o}rmander form. Setting $u\left( x,t\right) =v\left(
x,T-t\right) $ turns this into the initial value (Cauchy) problem, 
\begin{equation*}
du=\inf_{\alpha \in \mathcal{A}}\left[ b\left( x,\alpha \right)
Du+L_{a}u+f\left( x,\alpha \right) \right] dt+Du\cdot V\left( x\right) \circ
dB_{T-\cdot }
\end{equation*}%
with initial data $\,u\left( \cdot ,0\right) \equiv g$; and hence of a form
which is covered by theorem \ref{main} below. Indeed, $H=\left(
H_{1},H_{2}\right) $, $H_{i}\left( x,p\right) =p\cdot V_{i}\left( x\right) $%
, is linear in $p$. (Moreover, the rough driving signal in theorem \ref{main}
is taken as $\mathbf{z}_{t}:=\mathbf{B}_{T-t}\left( \omega \right) $ where $%
\mathbf{B}\left( \omega \right) $ is a fixed Brownian rough path, run
backwards in time.\footnote{%
Alternatively, the proof of theorem \ref{main} is trivially modified to
directly accomodate terminal data problems.})
\end{example}

Returning to the general setup of (\ref{PDE}), the results \cite{MR1647162,
MR1659958, MR1807189} are in fact \textit{pathwise} and apply to any
continuous path $z\in C\left( \left[ 0,T\right] ,\mathbb{R}^{d}\right) $,
this includes Brownian and even rougher sources of noise; however, the
assumption was made that $H=H\left( Du\right) $ is independent of $x$. The r%
\^{o}le of $x$-dependence is an important one (as it arises in applications
such as example \ref{ExStochControl}): the results of Lions--Souganidis
imply that the map%
\begin{equation*}
z\in C^{1}\left( \left[ 0,T\right] ,\mathbb{R}^{d}\right) \mapsto u\left(
\cdot ,\cdot \right) \in C\left( \left[ 0,T\right] ,\mathbb{R}^{n}\right)
\end{equation*}%
depends continuously on $z$ \textit{in uniform topology}; thereby giving
existence/uniqueness results to%
\begin{equation*}
du=F\left( t,x,Du,D^{2}u\right) dt+\sum_{i=1}^{d}H_{i}\left( Du\right) dz^{i}
\end{equation*}%
for \textit{every} continuous path $z:\left[ 0,T\right] \rightarrow \mathbb{R%
}^{d}$. When the Hamiltonian depends on $x$, this ceases to be true; indeed,
take $F\equiv 0,d=2$ and $H_{i}\left( x,p\right) =pV_{i}\left( x\right) $
where $V_{1},V_{2}$ are the vector fields from example \ref%
{ODEnotContInInfity}. Solving the characteristic equations shows that $u$ is
expressed in terms of the (inverse) flow associated to $dy=V_{1}\left(
y\right) dz^{1}+V_{2}\left( y\right) dz^{2}$, and we have already seen that
the solution of this ODE\ does not depend continuously on $z=\left(
z^{1},z^{2}\right) $ in uniform topology\footnote{%
We shall push this remark much further in theorem \ref{SPDEsussmann} below.}.

Of course, this type of problem can be prevented by strengthening the
topology: the Lyons' theory of rough paths does exhibit an entire cascade of
($p$-variation or $1/p$-H\"{o}lder type rough path) metrics (for each $p\geq
1$) on path-space under which such ODE solutions are continuous functions of
their driving signal. This suggests to extend the Lions--Souganidis theory
from a pathwise to a \textit{rough} pathwise theory. We shall do so for a
rich class of fully-nonlinear $F$ and Hamiltonians $H\left( x,Du\right) $
linear in $Du$. This last assumption allows for a global change of
coordinates which mimicks a classical trick in SPDE analysis (which, to the
best of our knowledge, goes back to Tubaro \cite{MR940902}, Kunita \cite[%
Chapter 6]{MR1472487} and Rozovski{\u{\i}} \cite{MR743902}, see also
Iftimie--Varsan \cite{MR1992898}; similar techniques have also proven useful
when $H=H\left( x,u\right) $ - we shall comment on this in section \ref%
{FurtherRmks}) where a SPDE is transformed into a random PDE (i.e. one that
can be solved with deterministic methods by fixing the randomness). In doing
so, the interplay between rough path and viscosity methods is illustrated in
a transparent way and everything boils down to combine the stability
properties of viscosity solution with those of differential equations in the
rough path sense. We have the following result.\footnote{%
Unless otherwise stated we shall always equip the spaces $\mathrm{BC}$ and $%
\mathrm{BUC}$ with the topology of locally uniform convergence.}\ 

\begin{theorem}
\label{main}Let $p\geq 1$ and $\left( z^{\varepsilon }\right) \subset
C^{\infty }\left( \left[ 0,T\right] ,\mathbb{R}^{d}\right) $ be Cauchy in ($%
p $-variation) rough path topology with rough path limit $\mathbf{z}\in
C^{0,p\text{-var}}\left( \left[ 0,T\right] ,G^{\left[ p\right] }\left( 
\mathbb{R}^{d}\right) \right) $. Assume%
\begin{equation*}
u_{0}^{\varepsilon }\in \mathrm{BUC}\left( \mathbb{R}^{n}\right) \rightarrow
u_{0}\in \mathrm{BUC}\left( \mathbb{R}^{n}\right) \text{,}
\end{equation*}%
locally uniformly as $\varepsilon \rightarrow 0$. Let $F=F\left(
t,x,p,X\right) $ be continuous, degenerate elliptic, and assume that $%
\partial _{t}-F$ satisfies $\Phi ^{\left( 3\right) }$-invariant comparison
(cf. definition \ref{PhiInvComp} below). Assume that $V=\left( V_{1},\dots
,V_{d}\right) $ is a collection of $\mathrm{Lip}^{\gamma +2}\left( \mathbb{R}%
^{n};\mathbb{R}^{n}\right) $ vector fields with $\gamma >p$. Assume
existence of (necessarily unique\footnote{%
This follows from the first 5 lines in the proof of this theorem.})
viscosity solutions $u^{\varepsilon }\in \mathrm{BC}\left( [0,T)\times 
\mathbb{R}^{n}\right) $ to%
\begin{eqnarray}
du^{\varepsilon } &=&F\left( t,x,Du^{\varepsilon },D^{2}u^{\varepsilon
}\right) dt-Du^{\varepsilon }\left( t,x\right) \cdot V\left( x\right)
dz^{\varepsilon }\left( t\right) ,  \label{espPDE} \\
\,\,\,u^{\varepsilon }\left( 0,\cdot \right) &=&u_{0}^{\varepsilon }
\end{eqnarray}%
and assume that the resulting family $\left( u^{\varepsilon }:\varepsilon
>0\right) $ is uniformly bounded\footnote{%
A simple sufficient conditions is boundedness of $F\left( \cdot ,\cdot
,0,0\right) $ on $\left[ 0,T\right] \times \mathbb{R}^{n}$, and the
assumption that $u_{0}^{\varepsilon }\rightarrow u_{0}$ uniformly, as can be
seen by comparison with function of the type $\left( t,x\right) \mapsto \pm
C\left( t+1\right) $, with sufficiently large $C.$%
\par
{}}. Then\newline
(i) there exists a unique $u\in \mathrm{BC}\left( [0,T)\times \mathbb{R}%
^{n}\right) $, only dependent on $\mathbf{z}$ and $u_{0}$ but not on the
particular approximating sequences, such that $u^{\varepsilon }\rightarrow u$
locally uniformly. We write (formally)%
\begin{eqnarray}
du &=&F\left( t,x,Du,D^{2}u\right) dt-Du\left( t,x\right) \cdot V\left(
x\right) d\mathbf{z}\left( t\right) ,  \label{RPDE} \\
\,\,\,u\left( 0,\cdot \right) &=&u_{0},
\end{eqnarray}%
and also $u=u^{\mathbf{z}}$ when we want to indicate the dependence on $%
\mathbf{z}$;\newline
(ii) we have the contraction property%
\begin{equation*}
\left\vert u^{\mathbf{z}}-\hat{u}^{\mathbf{z}}\right\vert _{\infty ;\mathbb{R%
}^{n}\times \left[ 0,T\right] }\leq \left\vert u_{0}-\hat{u}_{0}\right\vert
_{\infty ;\mathbb{R}^{n}}
\end{equation*}%
where $\hat{u}^{\mathbf{z}}$ is defined as limit of $\hat{u}^{\eta }$,
defined as in (\ref{espPDE}) with $u^{\varepsilon }$ replaced by $\hat{u}%
^{\eta }$ throughout;\newline
(iii) the solution map $\left( \mathbf{z,}u_{0}\right) \mapsto u^{\mathbf{z}%
} $ from 
\begin{equation*}
C^{p\text{-var}}\left( \left[ 0,T\right] ,G^{\left[ p\right] }\left( \mathbb{%
R}^{d}\right) \right) \times \mathrm{BUC}\left( \mathbb{R}^{n}\right)
\rightarrow \mathrm{BC}\left( [0,T)\times \mathbb{R}^{n}\right)
\end{equation*}%
is continuous.
\end{theorem}

Our proof actually allows for $\mathrm{BC}$ initial data in the above
theorem, since existence of solutions to the approximate problems (\ref%
{espPDE}) is assumed. Our preference for $\mathrm{BUC}$ initial data comes
from the fact that existence results are typically established under some
assumption of uniform continuity (e.g. \cite[Thm 3.1]{crandall-primer}).
Conversely, under a mild sharpening of the structural assumption which is
still satisfied in all our examples the solutions constructed in the above
theorem can be seen to be $\mathrm{BUC}$ ("bounded uniformly continuous") in
time-space. When $F=F\left( Du,D^{2}u\right) ,$ as in the setting of \cite%
{MR1647162, MR1807189}, a spatial modulus is easy to obtain; in the above
generality the comparison proof (based on doubling of the spatial variable)
can be adapted to obtain a spatial modulus of continuity, uniform in time
(this is implemented in \cite{MR1119185} for instance). Curiously, a modulus
in time cannot be established so directly; it is know however that a
"modulus of continuity in space" implies "modulus of continuity in time"
(cf. lemma 9.1. in \cite{MR1904498}). We shall return to such regularity
questions in detail in a separate note.

The reader may wonder if $u\in \mathrm{BC}\left( [0,T)\times \mathbb{R}%
^{n}\right) $ constructed in the above theorem solve a well-defined "rough"
PDE, apart from the formal equation (\ref{RPDE}). The answer is, in essence,
that $u$ is also a solution in the sense of Lions--Souganidis \cite%
{MR1647162, MR1659958, MR1807189} provided their definition is translated,
mutatis mutandis, to the present rough PDE setting. While we suspect that
such a point of view will be the key to a (rough) pathwise understanding of
fully non-linear stochastic partial differential equations, the present
situation ($H$ linear in $Du$) allows for a simpler understanding, still in
the spirit of Lions--Souganidis (to be specific, see \cite[Thm 2.4]%
{MR1647162}). The details of this are best given after the proof of theorem %
\ref{main}; we thus postpone further discussion on this to section \ref%
{AppSPDE}.

\begin{acknowledgement}
M. Caruana was supported by EPSRC\ grant EP/E048609/1. P.K. Friz has
received funding from the European Research Council under the European
Union's Seventh Framework Programme (FP7/2007-2013) / ERC grant agreement
nr. 258237, H. Oberhauser was supported by a DOC-fellowship of the Austrian
Academy of Sciences. Part of this work was undertaken while the last two
authors visited the Radon Institute (Austria). The authors would like to
thank Guy Barles for a helpful email exchange.
\end{acknowledgement}

\section{Condition for comparison}

We shall always assume that $F=F\left( t,x,p,X\right) $ is continuous and
degenerate elliptic. A sufficient condition\footnote{%
... which actually implies degenerate ellipticity, cf. page 18 in \cite[%
(3.14)]{MR1118699UserGuide}.} for comparison of (bounded) solutions to $%
\partial _{t}=F$ on $[0,T)\times \mathbb{R}^{n}$ is given by

\begin{condition}[{\protect\cite[(3.14)]{MR1118699UserGuide}}]
\label{UG314}$\ $There exists a function $\theta :\left[ 0,\infty \right]
\rightarrow \left[ 0,\infty \right] $ with $\theta \left( 0+\right) =0$,
such that for each fixed $t\in \left[ 0,T\right] $,%
\begin{equation}
F\left( t,x,\alpha \left( x-\tilde{x}\right) ,X\right) -F\left( t,\tilde{x}%
,\alpha \left( x-\tilde{x}\right) ,Y\right) \leq \theta \left( \alpha
\left\vert x-\tilde{x}\right\vert ^{2}+\left\vert x-\tilde{x}\right\vert
\right)  \label{UG314_1}
\end{equation}%
for all $\alpha >0$, $x,\tilde{x}\in \mathbb{R}^{n}$ and $X,Y\in S^{n}$ (the
space of $n\times n$ symmetric matrices) satisfy%
\begin{equation}
-3\alpha 
\begin{pmatrix}
I & 0 \\ 
0 & I%
\end{pmatrix}%
\leq 
\begin{pmatrix}
X & 0 \\ 
0 & -Y%
\end{pmatrix}%
\leq 3\alpha 
\begin{pmatrix}
I & -I \\ 
-I & I%
\end{pmatrix}%
\text{.}  \label{MatrixInequality}
\end{equation}%
Furthermore, we require $F=F\left( t,x,p,X\right) $ to be uniformly
continuous whenever $p,X$ remain bounded.
\end{condition}

Although this seems part of the folklore in viscosity theory\footnote{%
E.g. in Section 4.4. of Barles' 1997 lecture notes, www.phys.univ-tours.fr/%
\symbol{126}barles/Toulcours.pdf, or section V.9 in \cite{MR2179357FS}.}
only the case when $\mathbb{R}^{n}$ is replaced by a bounded domain is
discussed in the standard references (\cite[(3.14) and Section 8]%
{MR1118699UserGuide} or \cite[Section V.7, V.8]{MR2179357FS}; in this case
the very last requirement on uniform continuity can be omitted). For this
reason and the reader's convenience we have included a full proof of
parabolic comparison on $[0,T)\times \mathbb{R}^{n}$ under the above
condition in the appendix.

\begin{remark}[Stability under sup, inf etc]
\label{remark on sups}Using elementary inequalities of the type 
\begin{equation*}
\left\vert \sup \left( a\right) -\sup \left( b\right) \right\vert \leq \sup
\left\vert a-b\right\vert \text{ \ for }\,\,a,b\in \mathbb{R},
\end{equation*}%
one immediately sees that if $F_{\gamma },F_{\gamma ,\beta }$ satisfy (\ref%
{UG314_1}) for $\gamma ,\beta $ in some index set with a common modulus $%
\theta $, then $\inf_{\gamma }F_{\gamma },\sup_{\beta }\inf_{\gamma
}F_{\beta ,\gamma }$ etc again satisfy (\ref{UG314_1}). Similar remarks
apply to the uniform continuity property; provided there exists, for any $%
R<\infty $, a common modulus of continuity $\sigma _{R}$, valid whenever $%
p,X $ are of norm less than $R$.
\end{remark}

\section{Invariant comparison}

To motivate our key assumption on $F$ we need some preliminary remarks on
the transformation behaviour of 
\begin{equation*}
Du=\left( \partial _{1}u,\dots ,,\partial _{n}u\right) ,\,D^{2}u=\left(
\partial _{ij}u\right) _{i,j=1,\dots ,n}
\end{equation*}%
under change of coordinates on $\mathbb{R}^{n}$ where $u=u\left( t,\cdot
\right) $, for fixed $t$. Let us allow the change of coordinates to depend
on $t$, say $v\left( t,\cdot \right) :=u\left( t,\phi _{t}\left( \cdot
\right) \right) $ where $\phi _{t}:\mathbb{R}^{n}\rightarrow \mathbb{R}^{n}$
is a diffeomorphism. Differentiating $v\left( t,\phi _{t}^{-1}\left( \cdot
\right) \right) =u\left( t,\cdot \right) $ twice, followed by evaluation at $%
\phi _{t}\left( y\right) $, we have, with summation over repeated indices, 
\begin{eqnarray*}
\partial _{i}u\left( t,\phi _{t}\left( x\right) \right) &=&\partial
_{k}v\left( t,x\right) \partial _{i}\phi _{t}^{-1;k}|_{\phi _{t}\left(
x\right) } \\
\partial _{ij}u\left( t,\phi _{t}\left( x\right) \right) &=&\partial
_{kl}v\left( t,x\right) \partial _{i}\phi _{t}^{-1;k}|_{\phi _{t}\left(
x\right) }\partial _{j}\phi _{t}^{-1;l}|_{\phi _{t}\left( x\right)
}+\partial _{k}v\left( t,x\right) \partial _{ij}\phi _{t}^{-1;k}|_{\phi
_{t}\left( x\right) }.
\end{eqnarray*}%
We shall write this, somewhat imprecisely\footnote{%
Strictly speaking, one should view $\left( Du,D^{2}u\right) |_{\cdot }$ as 
\textit{second order} cotangent vector, the pull-back of $\left(
Dv,D^{2}v\right) |_{x}$ under $\phi _{t}^{-1}$.} but convenient, as 
\begin{eqnarray}
Du|_{\phi _{t}\left( x\right) } &=&\left\langle Dv|_{x},D\phi
_{t}^{-1}|_{\phi _{t}\left( x\right) }\right\rangle ,
\label{DuD2u_transformation} \\
\,\,D^{2}u|_{\phi _{t}\left( x\right) } &=&\left\langle D^{2}v|_{x},D\phi
_{t}^{-1}|_{\phi _{t}\left( x\right) }\otimes D\phi _{t}^{-1}|_{\phi
_{t}\left( x\right) }\right\rangle +\left\langle Dv|_{x},D^{2}\phi
_{t}^{-1}|_{\phi _{t}\left( x\right) }\right\rangle .  \notag
\end{eqnarray}

Let us now introduce $\Phi ^{\left( k\right) }$ as the class of all flows of 
$C^{k}$-diffeomorphisms of $\mathbb{R}^{n}$, $\phi =\left( \phi _{t}:t\in %
\left[ 0,T\right] \right) $, such that $\phi _{0}=\mathrm{Id}$ $\forall \phi
\in \Phi ^{\left( k\right) }$ and such that $\phi _{t}$ and $\phi _{t}^{-1}$
have $k$ bounded derivatives, uniformly in $t\in \left[ 0,T\right] $. We say
that $\phi \left( n\right) \rightarrow \phi $ in $\Phi ^{\left( k\right) }$
iff for all multi-indices $\alpha $ with $\left\vert \alpha \right\vert \leq
k$%
\begin{equation*}
\partial _{\alpha }\phi \left( n\right) \rightarrow \partial _{\alpha }\phi
_{t},\text{ }\partial _{\alpha }\phi \left( n\right) ^{-1}\rightarrow
\partial _{\alpha }\phi _{t}^{-1}\text{ locally uniformly in }\left[ 0,T%
\right] \times \mathbb{R}^{n}.\,\,
\end{equation*}

\begin{definition}[$\Phi ^{\left( k\right) }$-\textbf{invariant comparison; }%
$F^{\protect\phi }$]
\label{PhiInvComp}Let $k\geq 2$ and define $F^{\phi }\left( \left(
t,x,p,X\right) \right) $ as%
\begin{equation}
F\left( t,\phi _{t}\left( x\right) ,\left\langle p,D\phi _{t}^{-1}|_{\phi
_{t}\left( x\right) }\right\rangle ,\left\langle X,D\phi _{t}^{-1}|_{\phi
_{t}\left( x\right) }\otimes D\phi _{t}^{-1}|_{\phi _{t}\left( x\right)
}\right\rangle +\left\langle p,D^{2}\phi _{t}^{-1}|_{\phi _{t}\left(
x\right) }\right\rangle \right)  \label{definition of F_phi}
\end{equation}%
We say that $\partial _{t}=F$ satisfies $\Phi ^{\left( k\right) }$-invariant
comparison if, for every $\phi \in $ $\Phi ^{\left( k\right) }$, comparison
holds for bounded solutions of $\partial _{t}-F^{\phi }=0$. More precisely,
if $u$ is a bounded upper semi-continuous sub- and $v$ a bounded lower
semi-continuous super-solution to this equation and $u\left( 0,\cdot \right)
\leq v\left( 0,\cdot \right) $ then $u\leq v$ on $[0,T)\times \mathbb{R}^{n}$%
.
\end{definition}

\section{Examples}

\begin{example}[$F$ linear]
\label{linear}Suppose that $\sigma \left( t,x\right) :\left[ 0,T\right]
\times \mathbb{R}^{n}\rightarrow \mathbb{R}^{n\times n^{\prime }}$ and $%
b\left( t,x\right) :\left[ 0,T\right] \times \mathbb{R}^{n}\rightarrow 
\mathbb{R}^{n}$ are bounded, continuous in $t$ and Lipschitz continuous in $%
x $, uniformly in $t\in \left[ 0,T\right] $. If $F\left( t,x,p,X\right) =%
\mathrm{Tr}\left[ \sigma \left( t,x\right) \sigma \left( t,x\right) ^{T}X%
\right] +b\left( t,x\right) \cdot p$, then $\Phi ^{\left( 3\right) }$%
-invariant comparison holds. Although this is a special case of the
following example, let us point out that $F^{\phi }$ is of the same form as $%
F$ with $\sigma ,b$ replaced by%
\begin{eqnarray*}
\sigma ^{\phi }\left( t,x\right) _{m}^{k} &=&\sigma _{m}^{i}\left( t,\phi
_{t}\left( x\right) \right) \partial _{i}\phi _{t}^{-1;k}|_{\phi _{t}\left(
x\right) },\,\,\,k=1,\dots ,n;m=1,\dots ,n^{\prime } \\
b^{\phi }\left( t,x\right) ^{k} &=&\left[ b^{i}\left( t,\phi _{t}\left(
x\right) \right) \partial _{i}\phi _{t}^{-1;k}|_{\phi _{t}\left( x\right) }%
\right] +\sum_{i,j}\left( \sigma _{m}^{i}\sigma _{m}^{j}\partial _{ij}\phi
_{t}^{-1;k}|_{\phi _{t}\left( y\right) }\right) ,\,\,\,k=1,\dots ,n.
\end{eqnarray*}%
By defining properties of flows of diffeomorphisms, $t\mapsto \partial
_{i}\phi _{t}^{-1;k}|_{\phi _{t}\left( x\right) },\partial _{ij}\phi
_{t}^{-1;k}|_{\phi _{t}\left( y\right) }$ is continuous and the $C^{3}$%
-boundedness assumption inherent in our definition of $\Phi ^{\left(
3\right) }$ ensures that $\sigma ^{\phi },b^{\phi }$ are Lipschitz in $x$,
uniformly in $t\in \left[ 0,T\right] $. It is then easy to see (cf. the
argument of \cite[Lemma 7.1]{MR2179357FS}) that $F^{\phi }$ satisfies
condition \ref{UG314} for every $\phi \in \Phi ^{\left( 3\right) }$. This
implies that $\Phi ^{\left( 3\right) }$-invariant comparison holds for
bounded solutions of $\partial _{t}-F^{\phi }=0$.
\end{example}

\begin{example}[$F$ quasi-linear]
\label{quasilinear example}\bigskip Let 
\begin{equation}
F\left( t,x,p,X\right) =\mathrm{Tr}\left[ \sigma \left( t,x,p\right) \sigma
\left( t,x,p\right) ^{T}X\right] +b\left( t,x,p\right) .  \label{Fquasi}
\end{equation}%
We assume $b=b\left( t,x,p\right) :\left[ 0,T\right] \times \mathbb{R}%
^{n}\times \mathbb{R}^{n}\rightarrow \mathbb{R}$ is continuous, bounded and
Lipschitz continuous in $x$ and $p$, uniformly in $t\in \left[ 0,T\right] $.
We also assume that $\sigma =\sigma \left( t,x,p\right) :\left[ 0,T\right]
\times \mathbb{R}^{n}\times \mathbb{R}^{n}\rightarrow \mathbb{R}^{n\times
n^{\prime }}$ is a continuous, bounded map such that

\begin{itemize}
\item $\sigma \left( t,\cdot ,p\right) $ is Lipschitz continuous, uniformly
in $\left( t,p\right) \in \left[ 0,T\right] \times \mathbb{R}^{n}$;

\item there exists a constant $c_{1}>0$, such that\footnote{%
A condition of this type also appears also in \cite{MR2010963}.} 
\begin{equation}
\forall p,q\in \mathbb{R}^{n}:\left\vert \sigma \left( t,x,p\right) -\sigma
\left( t,x,q\right) \right\vert \leq c_{1}\frac{\left\vert p-q\right\vert }{%
1+\left\vert p\right\vert +\left\vert q\right\vert }  \label{eqSigmaGra}
\end{equation}%
for all $t\in \left[ 0,T\right] $ and $x\in \mathbb{R}^{n}$. \ 
\end{itemize}

We show that $F^{\phi }$ satisfies condition \ref{UG314} for every $\phi \in
\Phi ^{\left( 3\right) };$ this implies that $\Phi ^{\left( 3\right) }$%
-invariant comparison holds for $\partial _{t}=F$ with $F$ given by (\ref%
{Fquasi}). To see this we proceed as follows. For brevity denote 
\begin{eqnarray*}
p &=&\alpha \left( x-\tilde{x}\right) ,J_{\cdot }=D\phi _{t}^{-1}\lvert
_{\phi _{t}\left( \cdot \right) },H_{\cdot }=D^{2}\phi _{t}^{-1}\lvert
_{\phi _{t}\left( \cdot \right) } \\
\sigma _{\cdot } &=&\sigma \left( t,\phi _{t}\left( \cdot \right)
,\left\langle p,J_{\cdot }\right\rangle \right) ,a_{\cdot }=\sigma _{\cdot
}\sigma _{\cdot }^{T},\,b_{\cdot }=b\left( t,\phi _{t}\left( \cdot \right)
,\left\langle p,J_{\cdot }\right\rangle \right)
\end{eqnarray*}%
so that%
\begin{eqnarray*}
F^{\phi }\left( t,x,p,X\right) &=&\mathrm{Tr}\left[ a_{x}\left( \left\langle
X,J_{x}\otimes J_{x}\right\rangle +\left\langle p,H_{x}\right\rangle \right) %
\right] +b_{x} \\
&=&\mathrm{Tr}\left[ J_{x}a_{x}J_{x}^{T}X\right] +b_{x}+\mathrm{Tr}\left[
a_{x}\left\langle p,H_{x}\right\rangle \right] .
\end{eqnarray*}%
Hence 
\begin{equation*}
F^{\phi }\left( t,\tilde{x},p,Y\right) -F^{\phi }\left( t,x,p,X\right) =%
\underset{=:\left( i\right) }{\underbrace{\mathrm{Tr}\left[ J_{\tilde{x}}a_{%
\tilde{x}}J_{\tilde{x}}^{T}Y-J_{x}a_{x}J_{x}^{T}X\right] }}+\underset{%
=:\left( ii\right) }{\underbrace{b_{\tilde{x}}-b_{x}}}+\underset{=:\left(
iii\right) }{\underbrace{\mathrm{Tr}\left[ a_{\tilde{x}}\left\langle p,H_{%
\tilde{x}}\right\rangle -a_{x}\left\langle p,H_{x}\right\rangle \right] }}.
\end{equation*}%
To estimate $\left( i\right) $ note that $J_{x}a_{x}J_{x}^{T}=J_{x}\sigma
_{x}\left( J_{x}\sigma _{x}\right) ^{T}$. The $\mathbb{R}^{2n}\times \mathbb{%
R}^{2n}$ matrix 
\begin{equation*}
\left( 
\begin{array}{cc}
\left( J_{x}\sigma _{x}\right) \left( J_{x}\sigma _{x}\right) ^{T} & 
J_{x}\sigma _{x}\left( J_{\tilde{x}}\sigma _{\tilde{x}}\right) ^{T} \\ 
\left( J_{\tilde{x}}\sigma _{\tilde{x}}\right) \left( J_{x}\sigma
_{x}\right) ^{T} & J_{\tilde{x}}\sigma _{\tilde{x}}\left( J_{\tilde{x}%
}\sigma _{\tilde{x}}\right) ^{T}%
\end{array}%
\right)
\end{equation*}%
is positive semidefinite and thus we can multiply it to both sides of the
inequality%
\begin{equation*}
\left( 
\begin{array}{cc}
X & 0 \\ 
0 & -Y%
\end{array}%
\right) \leq 3\alpha \left( 
\begin{array}{cc}
I & -I \\ 
-I & I%
\end{array}%
\right) .
\end{equation*}%
The resulting inequality is stable under evaluating the trace and so one gets%
\begin{eqnarray*}
\mathrm{Tr}\left[ J_{x}\sigma _{x}\left( J_{x}\sigma _{x}\right) ^{T}\cdot
X-J_{\tilde{x}}\sigma _{\tilde{x}}\left( J_{\tilde{x}}\sigma _{\tilde{x}%
}\right) ^{T}\cdot Y\right] &\leq &3\alpha \mathrm{Tr}\left[ \left(
J_{x}\sigma _{x}\right) \left( J_{x}\sigma _{x}\right) ^{T}-J_{x}\sigma
_{x}\left( J_{\tilde{x}}\sigma _{\tilde{x}}\right) ^{T}\right. \\
&&\left. -J_{\tilde{x}}\sigma _{\tilde{x}}\left( J_{x}\sigma _{x}\right)
^{T}+J_{\tilde{x}}\sigma _{\tilde{x}}\left( J_{\tilde{x}}\sigma _{\tilde{x}%
}\right) ^{T}\right] \\
&=&3\alpha \mathrm{Tr}\left[ \left( J_{x}\sigma _{x}-J_{\tilde{x}}\sigma _{%
\tilde{x}}\right) \left( J_{x}\sigma _{x}-J_{\tilde{x}}\sigma _{\tilde{x}%
}\right) ^{T}\right] \\
&=&3\alpha \left\Vert J_{x}\sigma _{x}-J_{\tilde{x}}\sigma _{\tilde{x}%
}\right\Vert ^{2}
\end{eqnarray*}%
(using that $\mathrm{Tr}\left[ .\cdot .^{T}\right] $ defines an inner
product for matrices and gives rise to the Frobenius matrix norm $\left\Vert
.\right\Vert $). Hence, by the triangle inequality and Lipschitzness of the
Jacobian of the flow (which follows a fortiori from the boundedness of the
second order derivatives of the flow), 
\begin{eqnarray*}
\left\Vert J_{x}\sigma _{x}-J_{\tilde{x}}\sigma _{\tilde{x}}\right\Vert
&\leq &\left\Vert J_{x}\sigma _{x}-J_{x}\sigma _{\tilde{x}}\right\Vert
+\left\Vert J_{x}\sigma _{\tilde{x}}-J_{\tilde{x}}\sigma _{\tilde{x}%
}\right\Vert \\
&\leq &\left\Vert J_{x}\right\Vert \left\Vert \sigma _{x}-\sigma _{\tilde{x}%
}\right\Vert +\left\Vert J_{x}-J_{\tilde{x}}\right\Vert \left\Vert \sigma _{%
\tilde{x}}\right\Vert \\
&\leq &\left\Vert J_{x}\right\Vert \left\Vert \sigma _{x}-\sigma _{\tilde{x}%
}\right\Vert +c_{2}\left( \sigma ,\phi \right) \left\vert x-\tilde{x}%
\right\vert
\end{eqnarray*}%
Since $\sigma \left( t,\cdot ,q\right) $ is Lipschitz continuous (uniformly
in $\left( t,q\right) \in \left[ 0,T\right] \times \mathbb{R}^{n}$) and $%
\phi _{t}\left( \cdot \right) $ is Lipschitz continuous (uniformly in $t\in %
\left[ 0,T\right] $), we can use our assumption (\ref{eqSigmaGra})\ on $%
\sigma $, to see%
\begin{equation}
\left\Vert \sigma _{x}-\sigma _{\tilde{x}}\right\Vert \leq \left( \text{const%
}\right) \times \left\vert x-\tilde{x}\right\vert .
\label{AlsoUsefulForStep3}
\end{equation}%
Indeed, 
\begin{eqnarray*}
\left\Vert \sigma _{x}-\sigma _{\tilde{x}}\right\Vert &=&\left\Vert \sigma
\left( t,\phi _{t}\left( x\right) ,p\cdot J_{x}\right) -\sigma \left( t,\phi
_{t}\left( \tilde{x}\right) ,p\cdot J_{\tilde{x}}\right) \right\Vert \\
&\leq &\left\Vert \sigma \left( t,\phi _{t}\left( x\right) ,p\cdot
J_{x}\right) -\sigma \left( t,\phi _{t}\left( \tilde{x}\right) ,p\cdot
J_{x}\right) \right\Vert \\
&&+\left\Vert \sigma \left( t,\phi _{t}\left( \tilde{x}\right) ,p\cdot
J_{x}\right) -\sigma \left( t,\phi _{t}\left( \tilde{x}\right) ,p\cdot J_{%
\tilde{x}}\right) \right\Vert \\
&\leq &c_{2}\left( \sigma ,\phi \right) \left\vert x-\tilde{x}\right\vert
+c_{1}\frac{\alpha \left\vert x-\tilde{x}\right\vert \left\vert J_{x}-J_{%
\tilde{x}}\right\vert }{1+\alpha \left\vert x-\tilde{x}\right\vert \left(
\left\vert J_{x}\right\vert +\left\vert J_{\tilde{x}}\right\vert \right) };
\end{eqnarray*}%
and, noting that $\phi _{t}\circ \phi _{t}^{-1}=Id$ and $\sup_{\left(
t,x\right) \in \left[ 0,T\right] \times \mathbb{R}^{n}}\left\Vert D\phi
_{t}|_{x}\right\Vert \leq c_{3}$ implies $\left\Vert J_{x}\right\Vert
=\left\Vert D\phi _{t}^{-1}|_{\phi _{t}\left( x\right) }\right\Vert \geq
1/c_{3}$, we have 
\begin{eqnarray*}
c_{1}\frac{\alpha \left\vert x-\tilde{x}\right\vert \left\vert J_{x}-J_{%
\tilde{x}}\right\vert }{1+\alpha \left\vert x-\tilde{x}\right\vert \left(
\left\vert J_{x}\right\vert +\left\vert J_{\tilde{x}}\right\vert \right) }
&\leq &\left\vert x-\tilde{x}\right\vert \cdot \frac{c_{1}\alpha \left\vert
J_{x}-J_{\tilde{x}}\right\vert }{\alpha \left\vert x-\tilde{x}\right\vert
\left( \left\vert J_{x}\right\vert +\left\vert J_{\tilde{x}}\right\vert
\right) } \\
&\leq &\left\vert x-\tilde{x}\right\vert \frac{c_{4}\left( \sigma ,\phi
\right) \left\vert x-\tilde{x}\right\vert }{\left\vert x-\tilde{x}%
\right\vert \left( \left\vert J_{x}\right\vert +\left\vert J_{\tilde{x}%
}\right\vert \right) } \\
&\leq &c_{5}\left( \sigma ,\phi \right) \left\vert x-\tilde{x}\right\vert .
\end{eqnarray*}%
Putting things together we have%
\begin{equation*}
\left\vert \left( i\right) \right\vert \leq c_{6}\left( \sigma ,\phi \right)
\alpha \left\vert x-\tilde{x}\right\vert ^{2}.
\end{equation*}%
As for $\left( ii\right) $, we have that,%
\begin{eqnarray*}
\left\vert b_{x}-b_{\tilde{x}}\right\vert &\leq &\left\vert b\left( t,\phi
_{t}\left( x\right) ,\left\langle p,J_{x}\right\rangle \right) -b\left(
t,\phi _{t}\left( \tilde{x}\right) ,\left\langle p,J_{x}\right\rangle
\right) \right\vert \\
&&+\left\vert b\left( t,\phi _{t}\left( \tilde{x}\right) ,\left\langle
p,J_{x}\right\rangle \right) -b\left( t,\phi _{t}\left( \tilde{x}\right)
,\left\langle p,J_{\tilde{x}}\right\rangle \right) \right\vert \\
&\leq &c_{7}\left( b\right) \left( \left\vert \phi _{t}\left( x\right) -\phi
_{t}\left( \tilde{x}\right) \right\vert +\left\vert p\right\vert \left\vert
J_{\tilde{x}}-J_{x}\right\vert \right)
\end{eqnarray*}%
where $c_{7}\left( b\right) $ is the (uniform in $t\in \left[ 0,T\right] $)
Lipschitz bound for $b\left( t,\cdot ,\cdot \right) $. To get the required
estimate we again use the regularity of the flow. Finally, for $\left(
iii\right) $, 
\begin{eqnarray*}
\left( iii\right) &=&\mathrm{Tr}\left[ a_{\tilde{x}}\left\langle p,H_{\tilde{%
x}}\right\rangle -a_{\tilde{x}}\left\langle p,H_{x}\right\rangle \right] +%
\mathrm{Tr}\left[ a_{\tilde{x}}\left\langle p,H_{x}\right\rangle
-a_{x}\left\langle p,H_{x}\right\rangle \right] \\
&=&\mathrm{Tr}\left[ a_{\tilde{x}}\left\langle p,H_{\tilde{x}%
}-H_{x}\right\rangle \right] +\mathrm{Tr}\left[ (a_{\tilde{x}%
}-a_{x})\left\langle p,H_{x}\right\rangle \right] \text{.}
\end{eqnarray*}%
Using Cauchy-Schwartz (with inner product $\mathrm{Tr}\left[ .\cdot .^{T}%
\right] $) and $p=\alpha \left( x-\tilde{x}\right) $ it is clear that
boundedness of $H$ and $a$ (i.e. $\sup_{x}\left\vert H_{x}\right\vert
<\infty $ uniformly in $t\in \left[ 0,T\right] $ and similarly for $a$) and
Lipschitz continuity (i.e.$\left\vert H_{x}-H_{\tilde{x}}\right\vert \leq
\left( \text{const}\right) \times \left\vert x-\tilde{x}\right\vert $
uniformly in $t\in \left[ 0,T\right] $ and similar for $a$) will suffice to
obtain the (desired) estimate%
\begin{equation*}
\left\vert \left( iii\right) \right\vert \leq c_{8}\times \alpha \left\vert
x-\tilde{x}\right\vert ^{2}.
\end{equation*}%
Only Lipschitz continuity of $a_{x}=\sigma _{x}\sigma _{x}^{T}$ requires a
discsussion. But this follows, thanks to boundedness of $\sup_{x}\left\vert
\sigma _{x}\right\vert $, from showing Lipschitzness of $x\mapsto \sigma
_{x}=\sigma \left( t,\phi _{t}\left( x\right) ,\left\langle
p,J_{x}\right\rangle \right) $ uniformly in $t\in \left[ 0,T\right] $ which
was already seen in (\ref{AlsoUsefulForStep3}). This shows that $F^{\phi }$
satisfies (\ref{UG314_1}), for any $\phi \in \Phi ^{\left( 3\right) }$. To
see that $F^{\phi }$ satisfies condition \ref{UG314} it only remains to see
that $F^{\phi }\left( t,x,p,X\right) $ is uniformly continuous whenever $p,X$
remain bounded. To see this first observe that the flow map $\phi _{t}\left(
x\right) $, as function of $\left( t,x\right) \in \left[ 0,T\right] \times 
\mathbb{R}^{n}$, is uniformly continuous (but not bounded) while the
derivatives of the (inverse) flow, given by $J_{\cdot },H_{\cdot }$ above,
are bounded uniformly continuous maps as functions of $t,x$. One now easily
concludes with the fact the observations that (a) the product of \textrm{BUC}
function is again \textrm{BUC} and (b) the composition of a \textrm{BUC}
function with a \textrm{UC} function is again \textrm{BUC.}
\end{example}

\begin{example}[$F$ of Hamilton-Jacobi-Bellman type]
From the above examples and remark \ref{remark on sups}, we see that $\Phi
^{\left( 3\right) }$-invariant comparison holds when $F$ is given by%
\begin{equation*}
F\left( t,x,p,X\right) =\inf_{\gamma \in \Gamma }\left\{ \mathrm{Tr}\left[
\sigma \left( t,x;\gamma \right) \sigma \left( t,x;\gamma \right) ^{T}X%
\right] +b\left( t,x;\gamma \right) \cdot p\right\} ,
\end{equation*}%
the usual non-linearity in the Hamilton-Jacobi-Bellman equation, and more
generally 
\begin{equation*}
F\left( t,x,p,X\right) =\inf_{\gamma \in \Gamma }\left\{ \mathrm{Tr}\left[
\sigma \left( t,x,p;\gamma \right) \sigma \left( t,x,p;\gamma \right)
^{T}\cdot X\right] +b\left( t,x,p;\tau \right) \right\}
\end{equation*}%
whenever the conditions in examples \ref{linear} and \ref{quasilinear
example} are satisfied uniformly with respect to $\gamma \in \Gamma $. \ 
\end{example}

\begin{example}[$F$ of Isaac type]
Similarly, $\Phi ^{\left( 3\right) }$-invariant comparison holds for%
\begin{equation*}
F\left( t,x,p,X\right) =\sup_{\beta }\inf_{\gamma }\left\{ \mathrm{Tr}\left[
\sigma \left( t,x;\beta ,\gamma \right) \sigma \left( t,x;\beta ,\gamma
\right) ^{T}X\right] +b\left( t,x;\beta ,\gamma \right) \cdot p\right\} ,
\end{equation*}%
(such non-linearities arise in Isaac equation in the theory of differential
games), and more generally%
\begin{equation*}
F\left( t,x,p,X\right) =\sup_{\beta }\inf_{\gamma }\left\{ \mathrm{Tr}\left[
\sigma \left( t,x,p;\beta ,\gamma \right) \sigma \left( t,x,p;\beta ,\gamma
\right) ^{T}\cdot X\right] +b\left( t,x,p;\beta ,\gamma \right) \right\}
\end{equation*}%
whenever the conditions in examples \ref{linear} and \ref{quasilinear
example} are satisfied uniformly with respect to $\beta \in \mathcal{B}$ and 
$\gamma \in \Gamma $, where $\mathcal{B}$ and $\Gamma $ are arbitrary index
sets.
\end{example}

\section{Some lemmas}

\begin{lemma}
\label{ODE flow}\bigskip Let $z:\left[ 0,T\right] \rightarrow \mathbb{R}^{d}$
be smooth and assume that we are given $C^{3}$-bounded vector fields%
\footnote{%
In particular, if the vector fields are $\mathrm{Lip}^{\gamma }$, $\gamma
>p+2$, $p\geq 1$, then they are also $C^{3}$-bounded.} $V=\left( V_{1},\dots
,V_{d}\right) $. Then ODE%
\begin{equation*}
dy_{t}=V\left( y_{t}\right) dz_{t},\,\,\,t\in \left[ 0,T\right]
\end{equation*}%
has a unique solution flow (of $C^{3}$-diffeomorphisms) $\phi =\phi ^{z}\in
\Phi ^{\left( 3\right) }$.
\end{lemma}

\begin{proof}
Standard, e.g. chapter 4 in \cite{friz-victoir-book}.
\end{proof}

\begin{proposition}
\label{Prop solutions of F and F_phi}Let $z,V$ and $\phi $ be as in lemma %
\ref{ODE flow}. Then $u$ is a viscosity sub- (resp. super-) solution 
\begin{equation}
\dot{u}\left( t,x\right) =F\left( t,x,Du,D^{2}u\right) -Du\left( t,x\right)
\cdot V\left( x\right) \dot{z}\left( t\right)  \label{eq for u}
\end{equation}%
if and only if $v\left( t,x\right) :=u\left( t,\phi _{t}\left( x\right)
\right) $ is a viscosity sub- (resp. super-) solution of 
\begin{equation}
\dot{v}\left( t,x\right) =F^{\phi }\left( t,x,Dv,D^{2}v\right)
\label{eq for v}
\end{equation}%
where $F^{\phi }$ was defined in (\ref{definition of F_phi}).
\end{proposition}

\begin{proof}
Set $y=\phi _{t}\left( x\right) $. When $u$ is a classical sub-solution, it
suffices to use the the chain-rule and definition of $F^{\phi }$ to see that%
\begin{eqnarray*}
\dot{v}\left( t,x\right) &=&\dot{u}\left( t,y\right) +Du\left( t,y\right)
\cdot \dot{\phi}_{t}\left( x\right) =\dot{u}\left( t,y\right) +Du\left(
t,y\right) \cdot V\left( y\right) \dot{z}_{t} \\
&\leq &F\left( t,y,Du\left( t,y\right) ,D^{2}u\left( t,y\right) \right)
=F^{\phi }\left( t,x,Dv\left( t,x\right) ,D^{2}v\left( t,x\right) \right) .
\end{eqnarray*}%
The case when $u$ is a viscosity sub-solution of (\ref{eq for u}) is not
much harder: suppose that $\left( \bar{t},\bar{x}\right) $ is a maximum of $%
v-\xi $, where $\xi \in C^{2}\left( \left[ 0,T\right] \times \mathbb{R}%
^{n}\right) $ and define $\psi \in C^{2}\left( \left[ 0,T\right] \times 
\mathbb{R}^{n}\right) $ by $\psi \left( t,y\right) =\xi \left( t,\phi
_{t}^{-1}\left( y\right) \right) $. Set $\bar{y}=\phi _{\bar{t}}\left( \bar{x%
}\right) $ so that%
\begin{equation*}
F\left( \bar{t},\bar{y},D\psi \left( \bar{t},\bar{y}\right) ,D^{2}\psi
\left( \bar{t},\bar{y}\right) \right) =F^{\phi }\left( \bar{t},\bar{x},D\xi
\left( \bar{t},\bar{x}\right) ,D^{2}\xi \left( \bar{t},\bar{x}\right)
\right) .
\end{equation*}%
Obviously, $\left( \bar{t},\bar{y}\right) $ is a maximum of $u-\psi $, and
since $u$ is a viscosity sub-solution of (\ref{eq for u}) we have%
\begin{equation*}
\dot{\psi}\left( \bar{t},\bar{y}\right) +D\psi \left( \bar{t},\bar{y}\right)
V\left( \bar{y}\right) \dot{z}\left( \bar{t}\right) \leq F\left( \bar{t},%
\bar{y},D\psi \left( \bar{t},\bar{y}\right) ,D^{2}\psi \left( \bar{t},\bar{y}%
\right) \right) .
\end{equation*}%
On the other hand, $\xi \left( t,x\right) =\psi \left( t,\phi _{t}\left(
x\right) \right) $ implies $\dot{\xi}\left( \bar{t},\bar{x}\right) =\dot{\psi%
}\left( \bar{t},\bar{y}\right) +D\psi \left( \bar{t},\bar{y}\right) V\left( 
\bar{y}\right) \dot{z}\left( \bar{t}\right) $ and putting things together we
see that%
\begin{equation*}
\dot{\xi}\left( \bar{t},\bar{x}\right) \leq F^{\phi }\left( \bar{t},\bar{x}%
,D\xi \left( \bar{t},\bar{x}\right) ,D^{2}\xi \left( \bar{t},\bar{x}\right)
\right)
\end{equation*}%
which says precisely that $v$ is a viscosity sub-solution of (\ref{eq for v}%
). Replacing maximum by minimum and $\leq $ by $\geq $ in the preceding
argument, we see that if $u$ is a super-solution of (\ref{eq for u}), then $%
v $ is a super-solution of (\ref{eq for v}).\newline
Conversely, the same arguments show that if $v$ is a viscosity sub- (resp.
super-) solution for (\ref{eq for v}), then $u\left( t,y\right) =v\left(
t,\phi ^{-1}\left( y\right) \right) $ is a sub- (resp. super-) solution for (%
\ref{eq for u}).
\end{proof}

\section{Proof of the main result}

\begin{proof}
\textbf{(Theorem \ref{main}.)} Using Lemma \ref{ODE flow}, we see that $\phi
^{\varepsilon }\equiv \phi ^{z^{\varepsilon }}$, the solution flow to $%
dy=V\left( y\right) dz^{\varepsilon }$, is an element of $\Phi \equiv \Phi
^{\left( 3\right) }$. Set $F^{\varepsilon }:=F^{\phi ^{\varepsilon }}$. From
Proposition \ref{Prop solutions of F and F_phi}, we know that$\
u^{\varepsilon }$ is a solution to%
\begin{equation*}
du^{\varepsilon }=F\left( t,y,Du^{\varepsilon },D^{2}u^{\varepsilon }\right)
dt-Du^{\varepsilon }\left( t,y\right) \cdot V\left( y\right) dz^{\varepsilon
}\left( t\right) ,\,\,\,u^{\varepsilon }\left( 0,\cdot \right)
=u_{0}^{\varepsilon }
\end{equation*}%
if and only if $v^{\varepsilon }$ is a solution to $\partial
_{t}-F^{\varepsilon }=0$ with $v^{\varepsilon }\left( 0,\cdot \right)
=u_{0}^{\varepsilon }$. Let $\phi ^{\mathbf{z}}$ denote the solution flow to
the rough differential equation%
\begin{equation*}
dy=V\left( y\right) d\mathbf{z}.
\end{equation*}%
Thanks to $\mathrm{Lip}^{\gamma +2}$-regularity of the vector fields $\phi ^{%
\mathbf{z}}\in \Phi $, and in particular a flow of $C^{3}$-diffeomorphisms.
Set $F^{\mathbf{z}}=F^{\phi ^{\mathbf{z}}}$. The "universal" limit theorem 
\cite{lyons-qian-02} holds, in fact, on the level of flows of
diffeomorphisms (see \cite{lyons-qian-98} and \cite[Chapter 11]%
{friz-victoir-book} for more details) tells us that, since $z^{\varepsilon }$
tends to $\mathbf{z}$ in rough path sense,%
\begin{equation*}
\phi ^{\varepsilon }\rightarrow \phi ^{\mathbf{z}}\text{ in }\Phi
\end{equation*}%
so that, by continuity of $F$ (more precisely: uniform continuity on
compacts), we easily deduce that%
\begin{equation*}
F^{\varepsilon }\rightarrow F^{\mathbf{z}}\text{ locally uniformly.}
\end{equation*}%
From the method of semi-relaxed limits (Lemma 6.1 and Remarks 6.2, 6.3 and
6.4 in \cite{MR1118699UserGuide}, see also \cite{MR2179357FS}) the pointwise
(relaxed) limits%
\begin{eqnarray*}
\bar{v} &:&=\lim \sup \,^{\ast }\,\,v^{\varepsilon }, \\
\underline{v} &:&=\lim \inf \,_{\ast }\text{\thinspace \thinspace }%
v^{\varepsilon },
\end{eqnarray*}%
are USC sub- resp. LSC super-solutions to $\partial _{t}-F^{\mathbf{z}}=0$.
Boundedness of $\bar{v},\,\underline{v}$ is also clear by assumption that $%
\left( u^{\varepsilon }\right) $ is uniformly bounded. Moreover, since $%
v^{\varepsilon }\left( 0,\cdot \right) =u_{0}^{\varepsilon }\rightarrow
u_{0} $ locally uniformly as $\varepsilon \rightarrow 0$ it is not hard to
see that $\bar{v}\left( 0,\cdot \right) =\underline{v}\left( 0,\cdot \right)
=u_{0}$. (For the reader's convenience we have included a proof of this in
the appendix.) By assumption on $\Phi $-invariant comparison, the equation $%
\partial _{t}-F^{\mathbf{z}}=0$ satisfies comparison. It follows that $v:=%
\bar{v}=\underline{v}$ is the unique (and continuous, since $\bar{v},%
\underline{v}$ are upper resp. lower semi-continuous) solution to%
\begin{equation*}
\partial _{t}v=F^{\mathbf{z}}v\,,\,\,v\left( 0,\cdot \right) =u_{0}\left(
\cdot \right)
\end{equation*}%
(and hence that $v$ does not depend on the approximating sequence to $%
\mathbf{z}$). Moreover, using a simple Dini-type argument (e.g. \cite[p.35]%
{MR1118699UserGuide} or \cite[Lemme 4.1]{MR1613876}) one sees that this
limit must be uniform on compacts. The proof of (i) is finished by setting%
\begin{equation*}
u^{\mathbf{z}}\left( t,x\right) :=v\left( t,\left( \phi _{t}^{\mathbf{z}%
}\right) ^{-1}\left( x\right) \right) .
\end{equation*}%
(ii)\ The comparison $\left\vert u^{\mathbf{z}}-\hat{u}^{\mathbf{z}%
}\right\vert _{\infty ;\left[ 0,T\right] \times \mathbb{R}^{n}}\leq
\left\vert u_{0}-\hat{u}_{0}\right\vert _{\infty ;\mathbb{R}^{n}}$ is a
simple consequence of comparison for $v,\hat{v}$ (solutions to $\partial
_{t}v=F^{\mathbf{z}}v$). At last, to see (iii), we argue in the very same
way as in (i), starting with 
\begin{equation*}
F^{\mathbf{z}_{n}}\rightarrow F^{\mathbf{z}}\text{ locally uniformly}
\end{equation*}%
to see that $v^{n}\rightarrow v$ locally uniformly, i.e. uniformly on
compacts.
\end{proof}

\section{Applications to stochastic partial differential equations\label%
{AppSPDE}}

Applications to SPDEs are path-by-path, i.e. by taking $\mathbf{z}$ to be a
typical realization of Brownian motion and L\'{e}vy's area, $\mathbf{B}%
\left( \omega \right) \equiv \left( B,A\right) $, also known as enhanced
Brownian motion or Brownian rough path. The continuity property (iii) of our
theorem \ref{main} allows to identify (\ref{RPDE}) with $\mathbf{z}=\mathbf{B%
}\left( \omega \right) $ as \textit{Stratonovich solution} to the non-linear
SPDE%
\begin{equation*}
du=F\left( t,x,Du,D^{2}u\right) dt-Du\cdot V\left( x\right) \circ
dB,\,\,\,u\left( 0,\cdot \right) =u_{0}.
\end{equation*}%
Indeed, under the stated assumptions the Wong-Zakai approximations, in which
the Brownian $B$ is replaced by its piecewise linear approximation, based on
some mesh $\left\{ 0,\frac{T}{n},\frac{2T}{n}\dots ,T\right\} $, the
approximate solution will converge (locally uniformly on $\left[ 0,T\right]
\times \mathbb{R}^{n}$ and in probability, say) to the solution of%
\begin{equation*}
du=F\left( t,x,Du,D^{2}u\right) dt-Du\cdot V\left( x\right) d\mathbf{B}%
,\,\,\,u\left( 0,\cdot \right) =u_{0},
\end{equation*}%
as constructed in theorem \ref{main}. If one takes this piecewise linear
approximation property as \textit{definition} of a solution in Stratonovich
sense\footnote{%
... commonly done in the context of anticipating stochastic analysis, see 
\cite{Nu06, MR2319719} for instance.}\ this identification is trivially
settled. More interestingly, there is a number of Wong-Zakai approximation
results for SPDEs, ranging from \cite{MR1313027, MR1353194} to \cite%
{MR2052265, MR2268661}. Any solution of ours that is also covered in the
afore-mentioned references is then indeed a Stratonovich SPDE\ solution in
the usual sense\footnote{%
The same logic has been used by T. Lyons in \cite{lyons-98} to identify
rough differential equation driven by $\mathbf{B}$ as Stratonovich SDE
solutions.}. Of course, (Stratonovich) integral interpretations can break
down in degenerate situations. As example, consider non-differentiable
initial data $u_{0}$ and the (one-dimensional) random transport equation $%
du=u_{x}\circ dB$ with explicit "Stratonovich" solution $u_{0}\left(
x+B_{t}\right) $. (A similar situation occurs for the classical transport
equation $\dot{u}=u_{x}$, of course.) At last, we point out that our
solution also constitutes a \textit{stochastic viscosity solution} in the
sense of Lions--Souganidis \cite{MR1647162, MR1659958, MR1807189}: adapted
to the present setting, and recalling the notation used in the proof of
theorem\textbf{\ }\ref{main}, this amounts to call $u$ a (stochastic
viscosity) solution if $v\left( t,x\right) :=u\left( t,\left( \phi _{t}^{%
\mathbf{B}}\right) ^{-1}\left( x\right) \right) $ satisfies the (random)\
PDE $\partial _{t}v=F^{\mathbf{B}}v$ in viscosity sense\footnote{%
The actual definition of Lions--Souganidis is a localized version of this
and allows for noise of the form $H\left( x,Du\right) \circ dB$ with $H$
non-linear in $Du$. When $H$ is linear in $Du$, the standing assumption in
the present paper, the global and local definition are easily seen to be
equivalent.}. Observe that uniqueness of stochastic viscosity solutions then
follows from the classical theory of viscosity solutions of fully non-linear
second-order partial differential equations. (After all, our assumptions
guarantee that $\partial _{t}-F^{\mathbf{B}}$ satisfies comparison).

\begin{remark}[It\^{o} versus Stratonovich]
Note that similar \textbf{SPDEs in It\^{o}-form} need not be, in general,
well-posed. Consider the following (well-known) linear example 
\begin{equation*}
du=u_{x}dB+\lambda u_{xx}dt,\,\,\lambda \geq 0.
\end{equation*}%
A simple computation shows that $v\left( x,t\right) :=u\left(
x-B_{t},t\right) $ solves the (deterministic) PDE $\dot{v}=\left( \lambda
-1/2\right) v_{xx}$. From elementary facts about the heat-equation we
recognize that, for $\lambda <1/2$, this equation, with given initial data $%
v_{0}=u_{0}$, is not well-posed. In the (It\^{o}-)\ SPDE literature,
starting with \cite{MR553909}, this has led to coercivity conditions, also
known as super-parabolicity assumptions, in order to guarantee
well-posedness.
\end{remark}

\begin{remark}[Regularity of $V$]
Applied to the Brownian context (finite $p$-variation for any $p>2$) the
regularity assumption of theorem \ref{main} reads \textrm{Lip}$%
^{4+\varepsilon }$,$\varepsilon >0$. While our arguments do not appear to
leave much room for improvement we insist that working directly with
Stratonovich flows (rather than rough flows) will not bring much gain: the
regularity requirements are essentially the same. It\^{o} flows, on the
other hand, require one degree less in regularity. In turn, there is a
potential loss of well-posedness and the resulting SPDE is not robust as a
function of its driving noise (similar to classical It\^{o} stochastic
differential equations).
\end{remark}

\begin{remark}[Space-time regularity of SPDE\ solutions]
Since $u\left( t,x\right) =v\left( t,\phi _{t}^{\mathbf{B}}\left( x\right)
\right) $ and $\phi _{t}^{\mathbf{B}}$ is a flow of $C^{3}$- diffeomorphisms
the regularity of $u$ is readily reduced to regularity properties of $v$,
classical viscosity solution to $\partial _{t}v=F^{\mathbf{B}}v$. Unless one
make very specific assumptions on $F$ this is a difficult problem in its own
right; see the relevant remarks in \cite{MR1118699UserGuide} for instance.
\end{remark}

Let us now give some applications, typical in the sense that they have been
studied in great detail in the case of classical (Stratonovich) stochastic
differential equations.

\textbf{(Approximations)} \textit{Any} approximation result to $\mathbf{B}$
in rough path topology implies a corresponding (weak or strong) limit
theorem for such SPDEs: it suffices that an approximation to $B$ converges
in rough path topology; as is well known (e.g. \cite[Chapter 13]%
{friz-victoir-book} and the references therein) examples include piecewise
linear -, mollifier - and Karhunen-Loeve approximations, as well as (weak)
Donsker type random walk approximations \cite{BFH08}. The point being made,
we shall not spell out more details here.

\textbf{(Twisted approximations)} The following result implies en passant
that there is no (classical) pathwise theory of SPDEs in presence of spatial
dependence in the Hamilonian terms.

\begin{theorem}
\label{SPDEsussmann}Let $V=\left( V_{1},\dots ,V_{d}\right) $ be a
collection of $C^{\infty }$-bounded vector fields on $\mathbb{R}^{n}$ and $B$
a $d$-dimensional standard Brownian motion. Then, for every $\alpha =\left(
\alpha _{1},\dots ,\alpha _{N}\right) \in \left\{ 1,\dots ,d\right\} ^{N}$, $%
N\geq 2$, there exists (piecewise) smooth approximations $\left(
z^{k}\right) $ to $B$, with each $z^{k}$ only dependent on $\left\{ B\left(
t\right) :t\in D^{k}\right\} $ where $\left( D^{k}\right) $ is a sequence of
dissections of $\left[ 0,T\right] $ with mesh tending to zero, such that
almost surely%
\begin{equation*}
z^{k}\rightarrow B\text{ uniformly on }\left[ 0,T\right] ,
\end{equation*}%
but $u^{k}$, solutions to 
\begin{equation*}
du^{k}=F\left( t,x,Du^{k},D^{2}u^{k}\right) dt-Du^{k}\left( t,x\right) \cdot
V\left( x\right) dz^{k},\,\,\,u^{k}\left( 0,\cdot \right) =u_{0}\in \mathrm{%
BUC}\left( \mathbb{R}^{n}\right) ,
\end{equation*}%
(with assumptions on $F$ as formulated in theorem \ref{main}) converge
almost surely locally uniformly to the solution of the "wrong" differential
equation%
\begin{equation*}
du=\left[ F\left( t,x,Du,D^{2}u\right) -Du\left( t,x\right) \cdot V_{\alpha
}\left( x\right) \right] dt-Du\left( t,x\right) \cdot V\left( x\right) \circ
dB
\end{equation*}%
where $V_{\alpha }$ is the bracket-vector field given by $V_{\alpha }=\left[
V_{\alpha _{1}},\left[ V_{a_{2}},\dots \left[ V_{\alpha _{N-1}},V_{\alpha
_{N}}\right] \right] \right] $.
\end{theorem}

\begin{proof}
The rough path regularity of $\mathbf{B}\left( \omega \right) $ implies that
higher iterated (Stratonvich) integrals are deterministically defined; see 
\cite[First thm.]{lyons-98}. Doing this up to level $N$ yields a (rough
path) $S_{N}\left( \mathbf{B}\right) $ and we perturbe it in the highest
level, linearly in the 
\begin{equation*}
\left[ e_{\alpha _{1}},\left[ e_{a_{2}},\dots \left[ e_{\alpha
_{N-1}},e_{\alpha _{N}}\right] \right] \right] \text{-direction}
\end{equation*}%
of $S_{N}\left( \mathbf{B}\right) $ viewed as element in the step-$N$ free
nilpotent Lie algebra. This yields a (level-$N$) rough path $\mathbf{\tilde{B%
}}$ and we can find approximations $\left( z^{k}\right) $ that converge
almost surely to $\mathbf{\tilde{B}}$ in rough path topology (see \cite%
{friz-oberhauser-2008b}). One identifies standard RDEs driven by $\mathbf{%
\tilde{B}}$ as RDEs-with-drift (driven along the original vector fields by $d%
\mathbf{B}$, and along $V_{\alpha }$ by $dt$). The resulting identification
obviously holds on the level of RDE flows and thus 
\begin{equation*}
u^{z^{k}}\left( t,x\right) =v\left( t,\left( \phi _{t}^{z^{k}}\right)
^{-1}\left( x\right) \right) \rightarrow u^{\mathbf{\tilde{B}}}\left(
t,x\right) =v\left( t,\left( \phi _{t}^{\mathbf{\tilde{B}}}\right)
^{-1}\left( x\right) \right)
\end{equation*}

The flow identification then implies that%
\begin{equation*}
du=F\left( t,x,Du,D^{2}u\right) dt-Du\left( t,x\right) \cdot V\left(
x\right) d\mathbf{\tilde{B}}
\end{equation*}%
is equivalent to the equation with $V\left( x\right) d\mathbf{\tilde{B}}$
replaced by $V\left( x\right) d\mathbf{B+}V_{\alpha }\left( x\right) dt$.
\end{proof}

\begin{remark}
The attentive reader will have noticed that the preceding result also holds
when the Stratonovich differential$\,\circ dB$ is replaced by $dz$ for some $%
z\in C^{1}\left( \left[ 0,T\right] ,\mathbb{R}^{d}\right) $; it can then be
viewed as result on the effective behaviour of a (deterministic) non-linear
parabolic equations with coefficients that exhibit highly oscillatory
behaviour in time.
\end{remark}

\textbf{(Support results)} In conjunction with known support properties of $%
\mathbf{B}$ (e.g. \cite{LeQiZh02} in $p$-variation rough path topology or 
\cite{friz-lyons-stroock-06} for a conditional statement in H\"{o}lder rough
path topology) continuity of the SPDE solution as a function of $\mathbf{B}$
immediately implies Stroock--Varadhan type support descriptions for such
SPDEs. Let us note that, to the best of our knowledge, results of this type
are new for such non-linear SPDEs. In the linear case, approximations and
support of SPDEs have been studied in great detail \cite{MR1140746,
MR1026781, MR1019596, MR1011658, MR1008230}.

\textbf{(Large deviation results)} Another application of our continuity
result\ is the ability to obtain large deviation estimates when $B$ is
replaced by $\varepsilon B$ with $\varepsilon \rightarrow 0$; indeed, given
the known large deiviation behaviour of $\left( \varepsilon B,\varepsilon
^{2}A\right) $ in rough path topology (e.g. \cite{LeQiZh02} in $p$-variation
and\ \cite{friz-victoir-05} in H\"{o}lder rough path topology) it suffices
to recall that large deviation principles are stable under continuous maps.
Again, large deviation estimates for non-linear SPDEs in the small noise
limit appear to be new and may be hard to obtain without rough paths theory.%
\newline
\textbf{(SPDEs with non-Brownian noise)} Yet another benefit of our approach
is the ability to deal with SPDEs with non-Brownian and even
non-semimartingale noise. For instance, one can take $\mathbf{z}$ as (the
rough path lift of) fractional Brownian motion with Hurst parameter $%
1/4<H<1/2\,$, cf. \cite{coutin-qian-02} or \cite{friz-victoir-2007-gauss}, a
regime which is "rougher" than Brownian and notoriously difficult to handle;
or a diffusion with uniformly elliptic generator in divergence form with
measurable coefficients; see \cite{friz-victoir-subelliptic}. Much of the
above (approximations, support, large deviation) results also extend, as is
clear from the respective results in the above-cited literature.

\section{Further remarks\label{FurtherRmks}}

We have discussed a rough paths approach to stochastic partial differential
equation of the form%
\begin{equation*}
du=F\left( t,x,Du,D^{2}u\right) dt+\sum_{i=1}^{d}H_{i}\left( x,Du\right)
\circ dB^{i}
\end{equation*}%
with fully non-linear $F$ and Hamiltonian $H=H\left( x,Du\right) $, linear
in $Du$. When $F$ is (semi-)linear, uniformly elliptic, there are various
results (based on \textit{backward stochastic differential equations}) for
solving SPDE with general $H_{i}=H_{i}\left( u,Du,x\right) $; see \cite%
{MR1258986} for instance. Under a semi-linearity assumption $(F=\Delta
u+f\left( t,x,Du\right) $) and non-linear $H_{i}=H_{i}\left( Du,x\right) $,
subject to restrictive algebraic properties, a pathwise approach was carried
out by Iftimie--Varsan \cite{MR1992898}.

It is worth pointing out that SPDEs of the form%
\begin{equation*}
du=F\left( t,x,Du,D^{2}u\right) dt+\sum_{i=1}^{d}G_{i}\left( u,x\right)
\circ dB^{i},
\end{equation*}%
have also benefited from global (Doss--Sussmann) type transformations; see 
\cite{MR1799099, MR1828772, MR1831830}. Although this suggests that the
general rough path methodology of the present paper is also applicable for
such SPDEs, by considering rough PDEs of the form%
\begin{equation*}
du=F\left( t,x,Du,D^{2}u\right) dt+G\left( u,x\right) d\mathbf{z},
\end{equation*}%
the matter is far from straight-forward: a transformation based on the
(stochastic/rough) flow associsted to $G$, see \cite{MR1799099}, leads to a
transformed PDE which does not fit in the standard viscosity theory and we
shall return to this in future work.

\section{\protect\bigskip Appendix: comparison for parabolic equations}

Recall that $\mathrm{USC}$ (resp. $\mathrm{LSC}$) refers to upper (resp.
lower) semi-continuity. Let $u\in \mathrm{USC}\left( [0,T)\times \mathbb{R}%
^{n}\right) $ be a bounded subsolution to $\partial _{t}-F$; that is, $%
\partial _{t}u-F\left( t,x,Du,D^{2}u\right) \leq 0$ if $u$ is smooth and
with the usual viscosity definition otherwise. Similarly, let $v\in \mathrm{%
LSC}\left( [0,T)\times \mathbb{R}^{n}\right) $ be a bounded supersolution.

\begin{theorem}
Assume condition \ref{UG314}. Then comparison holds. That is,%
\begin{equation*}
u_{0}\leq v_{0}\text{ on }\mathbb{R}^{n}\implies u\leq v\text{ on }%
[0,T)\times \mathbb{R}^{n}.
\end{equation*}%
where $u_{0}=u\left( 0,\cdot \right) \in \mathrm{USC}\left( \mathbb{R}%
^{n}\right) $ and $v_{0}=v\left( 0,\cdot \right) \in \mathrm{LSC}\left( 
\mathbb{R}^{n}\right) $ denote the (bounded) initial data.
\end{theorem}

\begin{proof}
We follow the argument given in the User's Guide \cite[Section 8]%
{MR1118699UserGuide}. Without loss of generality, we may assume that $%
\partial _{t}u-F\left( t,x,Du,D^{2}u\right) \leq -c<0$ and that $%
\lim_{t\rightarrow T}u\left( t,x\right) =-\infty $ uniformly in $x\in 
\mathbb{R}^{n}$. We aim to contradict the existence of a point $\left(
s,z\right) \in \left( 0,T\right) \times \mathbb{R}^{n}$ such that%
\begin{equation*}
u\left( s,z\right) -v\left( s,z\right) =\delta >0\text{.}
\end{equation*}%
To this end, consider a maximum point $\left( \hat{t},\hat{x},\hat{y}\right)
\in \lbrack 0,T)\times \mathbb{R}^{n}\times \mathbb{R}^{n}$ of%
\begin{equation*}
\phi \left( t,x,y\right) =u\left( t,x\right) -v\left( t,y\right) -\frac{%
\alpha }{2}\left\vert x-y\right\vert ^{2}-\varepsilon \left( \left\vert
x\right\vert ^{2}+\left\vert y\right\vert ^{2}\right) .
\end{equation*}%
We first argue that, for small (resp. large) enough values of $\varepsilon $
and $\alpha $, the optimizing time paramter $\hat{t}\in \lbrack 0,T)$ cannot
be zero. Indeed, assuming $\hat{t}=0$ we can estimate%
\begin{eqnarray*}
\delta -2\varepsilon \left\vert z\right\vert ^{2} &=&\phi \left(
s,z,z\right)  \\
&\leq &\phi \left( 0,\hat{x},\hat{y}\right)  \\
&=&\sup_{x,y}\left[ u_{0}\left( x\right) -v_{0}\left( y\right) -\frac{\alpha 
}{2}\left\vert x-y\right\vert ^{2}-\varepsilon \left( \left\vert
x\right\vert ^{2}+\left\vert y\right\vert ^{2}\right) \right] .
\end{eqnarray*}%
From Lemma 3.1. in the User's Guide (applied to the $\mathrm{USC}\left( 
\mathbb{R}^{n}\right) $ resp. $\mathrm{LSC}\left( \mathbb{R}^{n}\right) $
map given by $u_{0}\left( x\right) -\varepsilon \left\vert x\right\vert ^{2}$
resp. $v_{0}\left( y\right) -\varepsilon \left\vert y\right\vert ^{2}$) it
follows that%
\begin{eqnarray*}
\lim_{\alpha \rightarrow \infty }\phi \left( 0,\hat{x},\hat{y}\right) 
&=&\sup_{x}\left[ u_{0}\left( x\right) -v_{0}\left( x\right) -2\varepsilon
\left\vert x\right\vert ^{2}\right]  \\
&\leq &\left\vert u_{0}-v_{0}\right\vert _{\infty ;\mathbb{R}^{n}}\leq 0%
\text{ by assumption.}
\end{eqnarray*}%
In particular, there exists $\alpha _{0}=\alpha _{0}\left( \delta \right) $
such that $\phi \left( 0,\hat{x},\hat{y}\right) <\delta /3$ for $\alpha \geq
\alpha _{0}$. If we then choose $\varepsilon \leq \varepsilon
_{0}=\varepsilon _{0}\left( \delta ,z\right) $, determined by $2\varepsilon
_{0}\left\vert z\right\vert ^{2}=\delta /3$ for instance, we are left with
the contradiction%
\begin{equation*}
\delta -\delta /3\leq \delta -2\varepsilon \left\vert z\right\vert ^{2}\leq
\phi \left( 0,\hat{x},\hat{y}\right) <\delta /3.
\end{equation*}%
It follows that $\hat{t}\in \left( 0,T\right) $ whenever $\varepsilon \leq
\varepsilon _{0}$ and $\alpha \geq \alpha _{0}$, which we shall assume from
here on. (In fact, we shall send $\varepsilon \rightarrow 0$, and then $%
\alpha \rightarrow \infty $, in what follows.)

Again, the plan is to arrive at a contradiction (so that we have to reject
the existence of a point $\left( s,z\right) \in \left( 0,T\right) \times 
\mathbb{R}^{n}$ at which $u\left( s,z\right) -v\left( s,z\right) >0$)
altogether. To this end, let us rewrite $\phi \left( t,x,y\right) $ as 
\begin{equation*}
\phi \left( t,x,y\right) =u^{\varepsilon }\left( t,x\right) -v^{\varepsilon
}\left( t,y\right) -\frac{\alpha }{2}\left\vert x-y\right\vert ^{2}
\end{equation*}%
where $u^{\varepsilon }\left( t,x\right) =u\left( t,x\right) -\varepsilon
\left\vert x\right\vert ^{2}$ and $v^{\varepsilon }\left( t,y\right)
=v\left( t,y\right) +\varepsilon \left\vert y\right\vert ^{2}$. Since $%
u^{\varepsilon }$ (resp. $v^{\varepsilon }$) are upper (resp. lower)
semi-continuous we can apply the (parabolic) theorem of sums \cite[Thm 8.3]%
{MR1118699UserGuide} at $\left( \hat{t},\hat{x},\hat{y}\right) $ to learn
that there are numbers $a,b$ and $X,Y\in S^{n}$ such that%
\begin{equation}
\left( a,\alpha \left( \hat{x}-\hat{y}\right) ,X\right) \in \mathcal{\bar{P}}%
^{2,+}u^{\varepsilon }\left( \hat{t},\hat{x}\right) ,\,\,\,\left( b,\alpha
\left( \hat{x}-\hat{y}\right) ,Y\right) \in \mathcal{\bar{P}}%
^{2,-}v^{\varepsilon }\left( \hat{t},\hat{y}\right)  \label{SemijetsForUVeps}
\end{equation}%
such that $a-b=0$ and such that one has the estimate (\ref{MatrixInequality}%
). It is easy to see (cf. \cite[Remark 2.7]{MR1118699UserGuide}) that (\ref%
{SemijetsForUVeps}) is equivalent to%
\begin{eqnarray*}
\text{ }\left( a,\alpha \left( \hat{x}-\hat{y}\right) +2\varepsilon \hat{x}%
,X+2\varepsilon I\right) &\in &\mathcal{\bar{P}}^{2,+}u\left( \hat{t},\hat{x}%
\right) , \\
\left( b,\alpha \left( \hat{x}-\hat{y}\right) -2\varepsilon \hat{y}%
,Y-2\varepsilon I\right) &\in &\mathcal{\bar{P}}^{2,-}v\left( \hat{t},\hat{y}%
\right) ;
\end{eqnarray*}%
using that $\partial _{t}u-F\left( t,x,Du,D^{2}u\right) \leq -c$ and $%
\partial _{t}v-F\left( t,x,Dv,D^{2}v\right) \geq 0$ (always understood in
the sense of viscosity sub- resp. super-solutions) we then see that%
\begin{eqnarray*}
a-F\left( \hat{t},\hat{x},\alpha \left( \hat{x}-\hat{y}\right) +2\varepsilon 
\hat{x},X+2\varepsilon I\right) &\leq &-c, \\
b-F\left( \hat{t},\hat{y},\alpha \left( \hat{x}-\hat{y}\right) -2\varepsilon 
\hat{y},Y-2\varepsilon I\right) &\geq &0.
\end{eqnarray*}%
Using $a=b$, this implies%
\begin{equation*}
0\leq c\leq F\left( \hat{t},\hat{x},\alpha \left( \hat{x}-\hat{y}\right)
+2\varepsilon \hat{x},X+2\varepsilon I\right) -F\left( \hat{t},\hat{y}%
,\alpha \left( \hat{x}-\hat{y}\right) -2\varepsilon \hat{y},Y-2\varepsilon
I\right) .
\end{equation*}%
The last step consists in showing that the right-hand-side converges to zero
by first sending $\varepsilon \rightarrow 0$ and then $\alpha \rightarrow
\infty $. (This yields the desired contradiction which ends the proof.) If $%
\varepsilon $ were absent (e.g. set $\varepsilon =0$ throughout) we would
estimate%
\begin{equation*}
F\left( \hat{t},\hat{x},\alpha \left( \hat{x}-\hat{y}\right) ,X\right)
-F\left( \hat{t},\hat{y},\alpha \left( \hat{x}-\hat{y}\right) ,Y\right) \leq
\theta \left( \alpha \left\vert \hat{x}-\hat{y}\right\vert ^{2}+\left\vert 
\hat{x}-\hat{y}\right\vert \right)
\end{equation*}%
and conclude (Lemma 3.1. in the User's Guide) that%
\begin{equation*}
\alpha \left\vert \hat{x}-\hat{y}\right\vert ^{2},\left\vert \hat{x}-\hat{y}%
\right\vert \rightarrow 0\text{ as }\alpha \rightarrow \infty
\end{equation*}%
in conjunction with continuity of $\theta $ at $0+$. The present case, $%
\varepsilon >0$, is essentially reduced to the case $\varepsilon =0$ by
adding/subtracting%
\begin{equation*}
F\left( \hat{t},\hat{x},\alpha \left( \hat{x}-\hat{y}\right) ,X\right)
-F\left( \hat{t},\hat{y},\alpha \left( \hat{x}-\hat{y}\right) ,Y\right) .
\end{equation*}%
It follows that $c\leq \left( i\right) +\left( ii\right) +\left( ii\right) $
where%
\begin{eqnarray*}
\left( i\right) &=&\left\vert F\left( \hat{t},\hat{x},\alpha \left( \hat{x}-%
\hat{y}\right) +2\varepsilon \hat{x},X+2\varepsilon I\right) -F\left( \hat{t}%
,\hat{x},\alpha \left( \hat{x}-\hat{y}\right) ,X\right) \right\vert \\
\left( ii\right) &=&\left\vert F\left( \hat{t},\hat{y},\alpha \left( \hat{x}-%
\hat{y}\right) -2\varepsilon \hat{y},Y-2\varepsilon I\right) -F\left( \hat{t}%
,\hat{y},\alpha \left( \hat{x}-\hat{y}\right) ,Y\right) \right\vert \\
\left( iii\right) &=&\theta \left( \alpha \left\vert \hat{x}-\hat{y}%
\right\vert ^{2}+\left\vert \hat{x}-\hat{y}\right\vert \right) .
\end{eqnarray*}%
From lemma \ref{LemmaPenalityOffDiag2} below, we see that (a) $p=\alpha
\left( \hat{x}-\hat{y}\right) $ remains, for fixed $\alpha $, bounded as $%
\varepsilon \rightarrow 0$, (b) $2\varepsilon \left\vert \hat{x}\right\vert $
and $2\varepsilon \left\vert \hat{y}\right\vert $ tend to zero as $%
\varepsilon \rightarrow 0$, for fixed $\alpha $, and (c) 
\begin{equation*}
\limsup_{\alpha \rightarrow \infty }\limsup_{\varepsilon \rightarrow
0}\left( \alpha \left\vert \hat{x}-\hat{y}\right\vert ^{2}+\left\vert \hat{x}%
-\hat{y}\right\vert \right) =\lim_{\alpha \rightarrow \infty
}\limsup_{\varepsilon \rightarrow 0}\left( \alpha \left\vert \hat{x}-\hat{y}%
\right\vert ^{2}+\left\vert \hat{x}-\hat{y}\right\vert \right) =0.
\end{equation*}%
We also note that (\ref{MatrixInequality}) implies that any matrix norm of $%
X,Y$ is bounded by a constant times $\alpha $, independent of $\varepsilon $%
. Since $F$ is assumed to be uniformly continuous whenever its gradient and
Hessian argument remain in abounded set, combining all this information
shows that 
\begin{equation*}
\limsup_{\varepsilon \rightarrow 0}\left( i\right) ,\,\limsup_{\varepsilon
\rightarrow 0}\left( ii\right) ,\,\limsup_{\varepsilon \rightarrow 0}\left(
iii\right)
\end{equation*}%
all tend to $0$ as $\alpha \rightarrow \infty $. In summary,%
\begin{equation*}
0<c\leq \lim_{\alpha \rightarrow \infty }\limsup_{\varepsilon \rightarrow 0} 
\left[ \left( i\right) +\left( ii\right) +\left( ii\right) \right] =0
\end{equation*}%
which is the desired contradiction. The proof is now finished.
\end{proof}

\begin{lemma}
\label{LemmaPenalityOffDiag2}Let \bigskip $u\in \mathrm{USC}\left(
[0,T)\times \mathbb{R}^{n}\right) $ bounded from above and $v\in \mathrm{LSC}%
\left( [0,T)\times \mathbb{R}^{n}\right) $ bounded from below. Consider a
maximum point $\left( \hat{t},\hat{x},\hat{y}\right) \in (0,T)\times \mathbb{%
R}^{n}\times \mathbb{R}^{n}$ of%
\begin{equation*}
\phi \left( t,x,y\right) =u\left( t,x\right) -v\left( t,y\right) -\frac{%
\alpha }{2}\left\vert x-y\right\vert ^{2}-\varepsilon \left( \left\vert
x\right\vert ^{2}+\left\vert y\right\vert ^{2}\right) .
\end{equation*}%
where $\alpha ,\varepsilon >0$. Then%
\begin{eqnarray}
\lim \sup_{\varepsilon \rightarrow 0}\alpha \left( \hat{x}-\hat{y}\right)
&=&C\left( \alpha \right) <\infty ,  \label{LemmaAppEst1} \\
\limsup_{\alpha \rightarrow \infty }\limsup_{\varepsilon \rightarrow
0}\varepsilon \left( \left\vert \hat{x}\right\vert ^{2}+\left\vert \hat{y}%
\right\vert ^{2}\right) &=&0,  \label{LemmaAppEst2} \\
\limsup_{\alpha \rightarrow \infty }\limsup_{\varepsilon \rightarrow
0}\left( \frac{\alpha }{2}\left\vert \hat{x}-\hat{y}\right\vert
^{2}+\left\vert \hat{x}-\hat{y}\right\vert \right) &=&0.
\label{LemmaAppEst3}
\end{eqnarray}
\end{lemma}

\begin{remark}
A similar lemma is found in \cite{bibKobylanski} or (without $t$ dependence) in Barles' book \cite[%
Lemme 4.3]{MR1613876}; the order in which limits are taken is important and
suggests the notation 
\begin{equation*}
\limsup_{\varepsilon <<\frac{1}{\alpha }\,\rightarrow 0}\,\left( ...\right)
:=\limsup_{\alpha \rightarrow \infty }\limsup_{\varepsilon \rightarrow
0}\,\left( ...\right) ,\,\,\liminf_{\varepsilon <<\frac{1}{\alpha }%
\,\rightarrow 0}\,\left( ...\right) :=\liminf_{\alpha \rightarrow \infty
}\liminf_{\varepsilon \rightarrow 0}\,\left( ...\right) .
\end{equation*}
\end{remark}

\begin{proof}
We start with some notation, where unless otherwise stated $t\in \left[ 0,T%
\right] $ and $x,y\in \mathbb{R}^{n}$,%
\begin{eqnarray*}
M_{\alpha ,\varepsilon } &:&=\sup_{t,x,y}\phi \left( t,x,y\right) =u\left( 
\hat{t},\hat{x}\right) -v\left( \hat{t},\hat{y}\right) -\frac{\alpha }{2}%
\left\vert \hat{x}-\hat{y}\right\vert ^{2}-\varepsilon \left( \left\vert 
\hat{x}\right\vert ^{2}+\left\vert \hat{y}\right\vert ^{2}\right) ; \\
M\left( h\right) &:&=\sup_{t,x,y:\left\vert x-y\right\vert \leq h}\left[
u\left( t,x\right) -v\left( t,y\right) \right] \geq \sup_{t,x}\left[ u\left(
t,x\right) -v\left( t,x\right) \right] \\
M^{\prime } &:&=\,\downarrow \lim_{h\rightarrow 0}M\left( h\right)
\end{eqnarray*}%
(As indicated, $M^{\prime }$ exists as limit of $M\left( h\right) $,
non-increasing in $h$ and bounded from below.)

\textbf{Step 1:} Take $t=x=y=0$ as argument of $\phi \left( t,x,y\right) $.
Since $M_{\alpha ,\varepsilon }=\sup \phi $ we have 
\begin{equation*}
c=u\left( 0,0\right) -v\left( 0,0\right) \leq M_{\alpha ,\varepsilon
}=u\left( \hat{t},\hat{x}\right) -v\left( \hat{t},\hat{y}\right) -\frac{%
\alpha }{2}\left\vert \hat{x}-\hat{y}\right\vert ^{2}-\varepsilon \left(
\left\vert \hat{x}\right\vert ^{2}+\left\vert \hat{y}\right\vert ^{2}\right)
\end{equation*}%
and hence, for a suitable constant $C$ (e.g. $C^{2}:=\sup u+\sup \left(
-v\right) +c$) 
\begin{equation*}
\frac{\alpha }{2}\left\vert \hat{x}-\hat{y}\right\vert ^{2}+\varepsilon
\left( \left\vert \hat{x}\right\vert ^{2}+\left\vert \hat{y}\right\vert
^{2}\right) \leq C^{2}
\end{equation*}%
which implies%
\begin{equation}
\left\vert \hat{x}-\hat{y}\right\vert \leq C\sqrt{2/\alpha }
\label{xhatMinusyhat}
\end{equation}%
and hence $\alpha \left\vert \hat{x}-\hat{y}\right\vert \leq \sqrt{2\alpha }%
C $ which is the first claimed estimate (\ref{LemmaAppEst1}).

\textbf{Step 2:} We first argue that it is enough to show the (two)
estimates 
\begin{equation}
\limsup_{\varepsilon <<\frac{1}{\alpha }\,\rightarrow 0}\left[ \,u\left( 
\hat{t},\hat{x}\right) -v\left( \hat{t},\hat{y}\right) \right] \leq
M^{\prime }\leq \liminf_{\varepsilon <<\frac{1}{\alpha }\,\rightarrow
0}\,M_{\alpha ,\varepsilon }.  \label{LemmaAppLeftToDo}
\end{equation}%
Indeed, from $\frac{\alpha }{2}\left\vert \hat{x}-\hat{y}\right\vert
^{2}+\varepsilon \left( \left\vert \hat{x}\right\vert ^{2}+\left\vert \hat{y}%
\right\vert ^{2}\right) =u\left( \hat{t},\hat{x}\right) -v\left( \hat{t},%
\hat{y}\right) -M_{\alpha ,\varepsilon }$ it readily follows that 
\begin{eqnarray*}
\limsup_{\varepsilon <<\frac{1}{\alpha }\,\rightarrow 0}\frac{\alpha }{2}%
\left\vert \hat{x}-\hat{y}\right\vert ^{2}+\varepsilon \left( \left\vert 
\hat{x}\right\vert ^{2}+\left\vert \hat{y}\right\vert ^{2}\right) &\leq
&\limsup_{\varepsilon <<\frac{1}{\alpha }\,\rightarrow 0}\left[ u\left( \hat{%
t},\hat{x}\right) -v\left( \hat{t},\hat{y}\right) -M_{\alpha ,\varepsilon }%
\right] \\
&=&\limsup_{\varepsilon <<\frac{1}{\alpha }\,\rightarrow 0}\left[ \,u\left( 
\hat{t},\hat{x}\right) -v\left( \hat{t},\hat{y}\right) \right]
-\liminf_{\varepsilon <<\frac{1}{\alpha }\,\rightarrow 0}\,M_{\alpha
,\varepsilon } \\
&\leq &0\text{ (and hence }=0\text{).}
\end{eqnarray*}%
This already gives (\ref{LemmaAppEst2}) and also (\ref{LemmaAppEst3}),
noting that%
\begin{equation*}
\left\vert \hat{x}-\hat{y}\right\vert =\alpha ^{-1/2}\alpha ^{1/2}\left\vert 
\hat{x}-\hat{y}\right\vert \leq \frac{1}{2\alpha }+\frac{\alpha }{2}%
\left\vert \hat{x}-\hat{y}\right\vert ^{2}.
\end{equation*}%
We are left to show (\ref{LemmaAppLeftToDo}). For the first estimate, it
suffices to note that, from (\ref{xhatMinusyhat}) and the definition of $%
M\left( h\right) $ applied with $h=C\sqrt{2/\alpha }$,%
\begin{eqnarray*}
\limsup_{\varepsilon <<\frac{1}{\alpha }\,\rightarrow 0}\left[ \,u\left( 
\hat{t},\hat{x}\right) -v\left( \hat{t},\hat{y}\right) \right] &\leq
&\limsup_{\varepsilon <<\frac{1}{\alpha }\,\rightarrow 0}M\left( \sqrt{\frac{%
2}{\alpha }}C\right) \\
&=&\lim_{\alpha \rightarrow \infty }M\left( \sqrt{\frac{2}{\alpha }}C\right)
=M^{\prime }.
\end{eqnarray*}%
We now turn to the second estimate in (\ref{LemmaAppLeftToDo}). From the
very definition of $M^{\prime }$ as $\lim_{h\rightarrow 0}M\left( h\right) $%
, there exists a family $\left( t_{h},x_{h},y_{h}\right) $ so that 
\begin{equation}
|x_{h}-y_{h}|\,\leq h\text{ and }u(t_{h},x_{h})-v(t_{h},x_{h})\rightarrow
M^{\prime }\text{ as }h\rightarrow 0  \label{eqBoundOnDistance}
\end{equation}%
For every $\alpha ,\varepsilon $ we may take $\left(
t_{h},x_{h},y_{h}\right) $ as argument of $\phi $; since $M_{\alpha
,\varepsilon }=\sup \phi $ we have 
\begin{equation}
u(t_{h},x_{h})-v(t_{h},y_{h})-\frac{\alpha }{2}h^{2}-\varepsilon
(|x_{h}|^{2}+|y_{h}|^{2})\leq M_{\alpha ,\varepsilon }  \label{eqSubseq}
\end{equation}%
Take now $\varepsilon =\varepsilon \left( h\right) \rightarrow 0$ with $%
h\rightarrow 0$; fast enough so that $\varepsilon
(|x_{h}|^{2}+|y_{h}|^{2})\rightarrow 0$; for instance $\varepsilon \left(
h\right) :=$ $h/\left( 1+(|x_{h}|^{2}+|y_{h}|^{2})\right) $ would do. It
follows that%
\begin{eqnarray*}
M^{\prime } &=&\lim_{h\rightarrow 0}u(t_{h},x_{h})-v(t_{h},y_{h}) \\
&=&\liminf_{h\,\rightarrow 0}u(t_{h},x_{h})-v(t_{h},y_{h})-\frac{\alpha }{2}%
h^{2}-\varepsilon (|x_{h}|^{2}+|y_{h}|^{2}) \\
&\leq &\liminf_{h\,\rightarrow 0}M_{\alpha ,\varepsilon
_{h}}=\liminf_{\varepsilon \,\rightarrow 0}M_{\alpha ,\varepsilon }\text{ by
monotonicity of }M_{\alpha ,\varepsilon }\text{ in }\varepsilon .
\end{eqnarray*}%
Since this is valid for every $\alpha $, we also have%
\begin{equation*}
M^{\prime }\leq \liminf_{\alpha \,\rightarrow \infty }\liminf_{\varepsilon
\,\rightarrow 0}M_{\alpha ,\varepsilon }.
\end{equation*}%
This is precisely the second estimate in (\ref{LemmaAppLeftToDo}) and so the
proof is finished.
\end{proof}

\section{\protect\bigskip Appendix: initial data under semi-relaxed limits}

Let $\left( v^{\varepsilon }:\varepsilon >0\right) $ denote a family of
uniformly bounded viscosity solutions $v^{\varepsilon }\left( t,x\right) $ to%
\begin{equation*}
\partial _{t}v^{\varepsilon }-F^{\varepsilon }\left( t,x,Dv^{\varepsilon
},D^{2}v^{\varepsilon }\right) =0,\,\,\,\,v^{\varepsilon }\left( 0,\cdot
\right) =g^{\varepsilon }\in \mathrm{BC}\left( \left[ 0,T\right] \times 
\mathbb{R}^{n}\right)
\end{equation*}%
where $F^{\varepsilon }=F^{\varepsilon }\left( t,x,p,X\right) $ is a
continuous function of its arguments. Assume $g^{\varepsilon }\rightarrow
g\in \mathrm{BC}$ locally uniformly and $F^{\varepsilon }\rightarrow F$
locally uniformly and recall that the semi-relaxed limits are defined by%
\begin{eqnarray*}
\bar{v}\left( t,x\right) &:&=\limsup_{\left( s,y\right) \in \left[ 0,T\right]
\times \mathbb{R}^{n}:s\rightarrow t,y\rightarrow x,\varepsilon \rightarrow
0}v^{\varepsilon }\left( s,y\right) , \\
\underline{v}\left( t,x\right) &:&=\liminf_{\left( s,y\right) \in \left[ 0,T%
\right] \times \mathbb{R}^{n}:s\rightarrow t,y\rightarrow x,\varepsilon
\rightarrow 0}v^{\varepsilon }\left( s,y\right) .
\end{eqnarray*}

\begin{proposition}
We have $\bar{v}\left( 0,x\right) =\underline{v}\left( 0,x\right) =g\left(
x\right) $.
\end{proposition}

\begin{proof}
We adapt the argument of Fleming--Soner \cite[Section VII.5]{MR2179357FS} to
our setting and focus on showing $\bar{v}\left( 0,x\right) =g\left( x\right) 
$, the other equality being similar. Trivially $\bar{v}\left( 0,x\right)
\geq g\left( x\right) $. Suppose equality does not hold. Then there exists $%
x_{0}\in \mathbb{R}^{n}$ and $\delta >0$ so that%
\begin{equation*}
\bar{v}\left( 0,x_{0}\right) =g\left( x_{0}\right) +\delta .
\end{equation*}%
We can assume without loss of generality $g\left( x_{0}\right) =0$; for
otherwise consider $\tilde{v}^{\varepsilon }\left( t,x\right)
:=v^{\varepsilon }\left( t,x\right) -g\left( x_{0}\right) $. Since $%
g^{\varepsilon }\rightarrow g$ uniformly near $x_{0}$ there are $\rho >0$
and $\varepsilon _{0}>0$ such that 
\begin{equation}
g^{\varepsilon }\left( x\right) <\delta /2\text{ \ \ \ whenever }\left\vert
x-x_{0}\right\vert ^{2}<\rho ,\,\,\,\varepsilon <\varepsilon _{0}\text{.}
\label{gepsLTdelta}
\end{equation}%
(We could take\ $\rho =1$ in fact.) Define the (smooth) test-function%
\begin{equation*}
w\left( t,x\right) =\gamma t+K\left\vert x-x_{0}\right\vert ^{2}
\end{equation*}%
where $K=(\sup_{\varepsilon >0}\left\vert v^{\varepsilon }\right\vert
_{\infty ;\left[ 0,T\right] \times \mathbb{R}^{n}}+1)/\rho $ and $\gamma
\geq K$ will be chosen later. Now, if $x$ is such that $g^{\varepsilon
}\left( x\right) \geq \delta /2$, and if $\varepsilon <\varepsilon _{0}$,
then (\ref{gepsLTdelta}) implies that we must have $\left\vert
x-x_{0}\right\vert ^{2}\geq \rho $; it then follows that%
\begin{equation*}
w\left( t,x\right) \geq K\left\vert x-x_{0}\right\vert ^{2}\geq K\rho \geq
v^{\varepsilon }\left( t,x\right) -g^{\varepsilon }\left( x_{0}\right) + 
\left[ 1+g^{\varepsilon }\left( x_{0}\right) \right]
\end{equation*}%
By making $\varepsilon _{0}$ smaller if necessary we can assume that $%
\left\vert g^{\varepsilon }\left( x_{0}\right) \right\vert <1/2,$ say, for
all $\varepsilon <\varepsilon _{0}$ which shows that%
\begin{equation}
w\left( t,x\right) >v^{\varepsilon }\left( t,x\right) -g^{\varepsilon
}\left( x_{0}\right) \text{ \ \ \ \ \ \ \ \ \ \ \ whenever \ }g^{\varepsilon
}\left( x\right) \geq \delta /2,\,\,\,\varepsilon <\varepsilon _{0}.
\label{WtxEstimate}
\end{equation}%
For $\varepsilon >0$, choose%
\begin{equation*}
\left( t_{\varepsilon },x_{\varepsilon }\right) \in \mathrm{argmax}\left\{
v^{\varepsilon }\left( t,x\right) -w\left( t,x\right) :\left( t,x\right) \in %
\left[ 0,T\right] \times \mathbb{R}^{n}\right\} .
\end{equation*}

Also the definition of $\bar{v}$ implies that there exists $\varepsilon
_{n}\rightarrow 0$ and $\left( s_{n},y_{n}\right) \rightarrow \left(
0,x_{0}\right) $ such that%
\begin{equation*}
\delta =\bar{v}\left( 0,x_{0}\right) =\lim_{n}v^{\varepsilon _{n}}\left(
s_{n},y_{n}\right) .
\end{equation*}%
Set $\left( t_{n},x_{n}\right) =\left( t_{\varepsilon _{n}},x_{\varepsilon
_{n}}\right) $. Then%
\begin{equation}
\liminf_{n\rightarrow \infty }v^{\varepsilon _{n}}\left( t_{n},x_{n}\right)
-w\left( t_{n},x_{n}\right) \geq v^{\varepsilon _{n}}\left(
s_{n},y_{n}\right) -w\left( s_{n},y_{n}\right) =\delta .
\label{liminfContra}
\end{equation}%
We claim that $t_{n}\neq 0$ for sufficiently large $n$. Indeed, if $t_{n}=0$
then $v^{\varepsilon _{n}}\left( t_{n},x_{n}\right) =g^{\varepsilon
_{n}}\left( x_{n}\right) $. The above inequality then yields that there is $%
n_{0}$ such that, for all $n\geq n_{0}$%
\begin{equation*}
g^{\varepsilon _{n}}\left( x_{n}\right) =v^{\varepsilon _{n}}\left(
t_{n},x_{n}\right) \geq \frac{\delta }{2}+w\left( t_{n},x_{n}\right) \geq 
\frac{\delta }{2}
\end{equation*}%
and, from (\ref{WtxEstimate}), $w\left( t_{n},x_{n}\right) >v^{\varepsilon
_{n}}\left( t_{n},x_{n}\right) -g^{\varepsilon _{n}}\left( x_{0}\right) $.
It follows that%
\begin{equation*}
\liminf_{n\rightarrow \infty }v^{\varepsilon _{n}}\left( t_{n},x_{n}\right)
-w\left( t_{n},x_{n}\right) \leq \liminf_{n\rightarrow \infty
}g^{\varepsilon _{n}}\left( x_{0}\right) =0
\end{equation*}%
which contradicts (\ref{liminfContra}). Hence, for all $n\geq n_{0}$ we have 
$t_{n}\neq 0$ and the viscosity property of $v^{\varepsilon }$ gives (with
all derivatives evaluated at $t_{n},x\,_{n}$)%
\begin{eqnarray*}
0 &\geq &\partial _{t}w\left( t_{n},x_{n}\right) -F^{\varepsilon _{n}}\left(
t_{n},x_{n},Dw,D^{2}w\right) \\
&=&\gamma -F^{\varepsilon _{n}}\left( t_{n},x_{n},2K\left(
x_{n}-x_{0}\right) ,2KI\right) .
\end{eqnarray*}

Note that $x_{n}=x_{n}^{\gamma }$ (to indicate the dependence on $\gamma $)
remains bounded, uniformly in $n\in \left\{ n_{0},\dots \right\} $ and $%
\gamma \in \lbrack K,\infty )$. Since $F^{\varepsilon _{n}}$ is locally
uniformly continuous it is also locally uniformly bounded. A contradiction
is now obtained by taking $\gamma $ large enough.
\end{proof}


\begin{thebibliography}{99}
\bibitem{MR1613876} Guy Barles, 
\newblock {\em Solutions de viscosit\'e des
\'equations de {H}amilton-{J}acobi}. \newblock Springer, 2004.

\bibitem{MR2010963} Guy Barles, Samuel Biton, Mariane Bourgoing, and Olivier
Ley. \newblock Uniqueness results for quasilinear parabolic equations
through viscosity solutions' methods. 
\newblock {\em Calc. Var. Partial
Differential Equations}, 18(2):159--179, 2003.

\bibitem{MR1904498} Guy Barles, Samuel Biton, Olivier Ley. \newblock A
geometrical approach to the study of unbounded solutions of quasilinear
parabolic equations. 
\newblock {\em Archive for Rational Mechanics and
Analysis}, 162, no. 4, 287--325, 2002.

\bibitem{BFH08} Emmanuel Breuillard, Peter Friz, and Martin Huesmann. %
\newblock From random walks to rough paths. 
\newblock {\em Proc. Amer. Math.
Soc.} 137 (2009), 3487-3496

\bibitem{MR1313027} Zdzis{\l }aw Brze{\'{z}}niak and Franco Flandoli. %
\newblock Almost sure approximation of {W}ong-{Z}akai type for stochastic
partial differential equations. \newblock {\em Stochastic Process. Appl.},
55(2):329--358, 1995.

\bibitem{MR1828772} Rainer Buckdahn and Jin Ma. \newblock Stochastic
viscosity solutions for nonlinear stochastic partial differential equations. 
{I}. \newblock {\em Stochastic Process. Appl.}, 93(2):181--204, 2001.

\bibitem{MR1831830} Rainer Buckdahn and Jin Ma. \newblock Stochastic
viscosity solutions for nonlinear stochastic partial differential equations. 
{II}. \newblock {\em Stochastic Process. Appl.}, 93(2):205--228, 2001.

\bibitem{MR1920103} Rainer Buckdahn and Jin Ma. \newblock Pathwise
stochastic {T}aylor expansions and stochastic viscosity solutions for fully
nonlinear stochastic {PDE}s. \newblock {\em Ann. Probab.}, 30(3):1131--1171,
2002.

\bibitem{MR2285722} Rainer Buckdahn and Jin Ma. \newblock Pathwise
stochastic control problems and stochastic {HJB} equations. 
\newblock {\em
SIAM J. Control Optim.}, 45(6):2224--2256 (electronic), 2007.

\bibitem{MR2319719} Laure Coutin, Peter Friz, and Nicolas Victoir. \newblock %
Good rough path sequences and applications to anticipating stochastic
calculus. \newblock {\em Ann. Probab.}, 35(3):1172--1193, 2007.

\bibitem{coutin-qian-02} Laure Coutin and Zhongmin Qian. \newblock %
Stochastic analysis, rough path analysis and fractional {B}rownian motions. %
\newblock {\em Probab. Theory Related Fields}, 122(1):108--140, 2002.

\bibitem{crandall-primer} Michael~G. Crandall, \newblock Viscosity
solutions: A Primer. \newblock {\em LNM} 1660, 1995.

\bibitem{MR1118699UserGuide} Michael~G. Crandall, Hitoshi Ishii, and
Pierre-Louis Lions. \newblock User's guide to viscosity solutions of second
order partial differential equations. 
\newblock {\em Bull. Amer. Math. Soc.
(N.S.)}, 27(1):1--67, 1992.

\bibitem{MR1275134} Mark H.~A. Davis and Gabriel Burstein. \newblock A
deterministic approach to stochastic optimal control with application to
anticipative control. \newblock {\em Stochastics Stochastics Rep.},
40(3-4):203--256, 1992.

\bibitem{MR2179357FS} Wendell~H. Fleming and H.~Mete Soner. 
\newblock {\em
Controlled {M}arkov processes and viscosity solutions}, volume~25 of \emph{%
Stochastic Modelling and Applied Probability}. \newblock Springer, New York,
second edition, 2006.

\bibitem{friz-lyons-stroock-06} P.~Friz, T.~Lyons, and D.~Stroock. \newblock %
L\'{e}vy's area under conditioning. 
\newblock {\em Ann. Inst. H. Poincar\'e
Probab. Statist.}, 42(1):89--101, 2006.

\bibitem{friz-oberhauser-2008b} Peter Friz and Harald Oberhauser. \newblock %
Rough path limits of the {Wo}ng-{Z}akai type with a modified drift term. %
\newblock {\em J. Funct. Anal.} \newblock, 256, pp. 3236-3256, 2009.

\bibitem{friz-victoir-2007-gauss} Peter Friz and Nicolas Victoir. \newblock %
Differential equations driven by {G}aussian signals. 
\newblock {\em Ann. Inst. H. Poincar\'e Probab. Statist.},
46(2):369--413, 2010.

\bibitem{friz-victoir-05} Peter Friz and Nicolas Victoir. \newblock %
Approximations of the {B}rownian rough path with applications to stochastic
analysis. \newblock {\em Ann. Inst. H. Poincar\'e Probab. Statist.},
41(4):703--724, 2005.

\bibitem{friz-victoir-subelliptic} Peter Friz and Nicolas Victoir. \newblock %
On uniformly subelliptic operators and stochastic area. 
\newblock {\em
Probab. Theory Related Fields}, 142(3-4):475--523, 2008.

\bibitem{friz-victoir-book} Peter~K. Friz and Nicolas~B. Victoir. 
\newblock {\em Multidimensional stochastic processes as rough paths: theory and
  applications}. \newblock Cambridge Studies in Advanced Mathematics, 120.
Cambridge University Press, Cambridge, 2010. \newblock

\bibitem{MR1119185} Giga, Y.; Goto, S.; Ishii, H.; Sato, M.-H. \newblock %
Comparison principle and convexity preserving properties for singular
degenerate parabolic equations on unbounded domains. 
\newblock {\em Indiana
Univ. Math. J}, 40, no. 2, 443--470, 1991.

\bibitem{MR1011658} I.~Gy{\"{o}}ngy. \newblock The stability of stochastic
partial differential equations and applications. {I}. 
\newblock {\em
Stochastics Stochastics Rep.}, 27(2):129--150, 1989.

\bibitem{MR1019596} I.~Gy{\"{o}}ngy. \newblock The stability of stochastic
partial differential equations and applications. {T}heorems on supports. %
\newblock In \emph{Stochastic partial differential equations and
applications, {II} ({T}rento, 1988)}, volume 1390 of \emph{Lecture Notes in
Math.}, pages 91--118. Springer, Berlin, 1989.

\bibitem{MR1008230} I.~Gy{\"{o}}ngy. \newblock The stability of stochastic
partial differential equations. {II}. 
\newblock {\em Stochastics Stochastics
Rep.}, 27(3):189--233, 1989.

\bibitem{MR1026781} I.~Gy{\"{o}}ngy. \newblock The approximation of
stochastic partial differential equations and applications in nonlinear
filtering. \newblock {\em Comput. Math. Appl.}, 19(1):47--63, 1990.

\bibitem{MR1140746} Istv{\'{a}}n Gy{\"{o}}ngy. \newblock On stochastic
partial differential equations. {R}esults on approximations. \newblock In 
\emph{Topics in stochastic systems: modelling, estimation and adaptive
control}, volume 161 of \emph{Lecture Notes in Control and Inform. Sci.},
pages 116--136. Springer, Berlin, 1991.

\bibitem{MR2052265} Istv{\'{a}}n Gy{\"{o}}ngy and Gy{\"{o}}rgy Michaletzky. %
\newblock On {W}ong-{Z}akai approximations with {$\delta $}-martingales. %
\newblock {\em Proc. R. Soc. Lond. Ser. A Math. Phys. Eng. Sci.},
460(2041):309--324, 2004. \newblock Stochastic analysis with applications to
mathematical finance.

\bibitem{MR2268661} Istv{\'{a}}n Gy{\"{o}}ngy and Anton Shmatkov. \newblock %
Rate of convergence of {W}ong-{Z}akai approximations for stochastic partial
differential equations. \newblock {\em Appl. Math. Optim.}, 54(3):315--341,
2006.

\bibitem{MR1992898} Bogdan Iftimie and Constantin Varsan. \newblock A
pathwise solution for nonlinear parabolic equations with stochastic
perturbations. \newblock {\em Cent. Eur. J. Math.}, 1(3):367--381
(electronic), 2003.

\bibitem{bibKobylanski}
M.~Kobylanski.
\newblock Backward stochastic differential equations and partial differential
  equations with quadratic growth.
\newblock {\em Ann. Probab.}, 28(2):558--602, 2000.

\bibitem{MR1472487} Hiroshi Kunita. 
\newblock {\em Stochastic flows and
stochastic differential equations}, volume~24 of \emph{Cambridge Studies in
Advanced Mathematics}. \newblock Cambridge University Press, Cambridge,
1997. \newblock Reprint of the 1990 original.

\bibitem{LeQiZh02} M.~Ledoux, Z.~Qian, and T.~Zhang. \newblock Large
deviations and support theorem for diffusion processes via rough paths. %
\newblock {\em Stochastic Process. Appl.}, 102(2):265--283, 2002.

\bibitem{MR1959710} P.-L. Lions and P.~E. Souganidis. \newblock Viscosity
solutions of fully nonlinear stochastic partial differential equations. %
\newblock {\em S\=urikaisekikenky\=usho K\=oky\=uroku}, (1287):58--65, 2002. %
\newblock Viscosity solutions of differential equations and related topics
(Japanese) (Kyoto, 2001).

\bibitem{MR1647162} Pierre-Louis Lions and Panagiotis~E. Souganidis. %
\newblock Fully nonlinear stochastic partial differential equations. %
\newblock {\em C. R. Acad. Sci. Paris S\'er. I Math.}, 326(9):1085--1092,
1998.

\bibitem{MR1659958} Pierre-Louis Lions and Panagiotis~E. Souganidis. %
\newblock Fully nonlinear stochastic partial differential equations:
non-smooth equations and applications. 
\newblock {\em C. R. Acad. Sci. Paris
S\'er. I Math.}, 327(8):735--741, 1998.

\bibitem{MR1799099} Pierre-Louis Lions and Panagiotis~E. Souganidis. %
\newblock Fully nonlinear stochastic pde with semilinear stochastic
dependence. \newblock {\em C. R. Acad. Sci. Paris S\'er. I Math.},
331(8):617--624, 2000.

\bibitem{MR1807189} Pierre-Louis Lions and Panagiotis~E. Souganidis. %
\newblock Uniqueness of weak solutions of fully nonlinear stochastic partial
differential equations. \newblock {\em C. R. Acad. Sci. Paris S\'er. I Math.}%
, 331(10):783--790, 2000.

\bibitem{lyons-98} Terry Lyons. \newblock Differential equations driven by
rough signals. \newblock {\em Rev. Mat. Iberoamericana}, 14(2):215--310,
1998.

\bibitem{lyons-qian-98} Terry Lyons and Zhongmin Qian. \newblock Flow of
diffeomorphisms induced by a geometric multiplicative functional. \newblock
\emph{Probab. Theory Related Fields}, 112(1):91--119, 1998.

\bibitem{lyons-qian-02} Terry Lyons and Zhongmin Qian. 
\newblock {\em System
{C}ontrol and {R}ough {P}aths}. \newblock Oxford University Press, 2002. %
\newblock Oxford Mathematical Monographs.

\bibitem{MR2314753} Terry~J. Lyons, Michael Caruana, and Thierry L{\'{e}}vy. %
\newblock {\em Differential equations driven by rough paths}, volume 1908 of 
\emph{Lecture Notes in Mathematics}. \newblock Springer, Berlin, 2007. %
\newblock Lectures from the 34th Summer School on Probability Theory held in
Saint-Flour, July 6--24, 2004, With an introduction concerning the Summer
School by Jean Picard.

\bibitem{Nu06} David Nualart. 
\newblock {\em The {M}alliavin calculus and
related topics}. \newblock Probability and its Applications (New York).
Springer-Verlag, Berlin, second edition, 2006.

\bibitem{MR553909} E.~Pardoux. \newblock Stochastic partial differential
equations and filtering of diffusion processes. \newblock {\em Stochastics},
3(2):127--167, 1979.

\bibitem{MR1258986} {\'{E}}tienne Pardoux and Shi~Ge Peng. \newblock %
Backward doubly stochastic differential equations and systems of quasilinear 
{SPDE}s. \newblock {\em Probab. Theory Related Fields}, 98(2):209--227, 1994.

\bibitem{MR743902} B..~L. Rozovski{\u{\i}}. 
\newblock {\em Evolyutsionnye
stokhasticheskie sistemy}. \newblock\textquotedblleft
Nauka\textquotedblright , Moscow, 1983. \newblock Lineinaya teoriya i
prilozkheniya k statistike sluchainykh protsessov. [Linear theory and
applications to the statistics of random processes].

\bibitem{MR940902} Luciano Tubaro. \newblock Some results on stochastic
partial differential equations by the stochastic characteristics method. %
\newblock {\em Stochastic Anal. Appl.}, 6(2):217--230, 1988.

\bibitem{MR1353194} Krystyna Twardowska. \newblock An approximation theorem
of {W}ong-{Z}akai type for nonlinear stochastic partial differential
equations. \newblock {\em Stochastic Anal. Appl.}, 13(5):601--626, 1995.
\end{thebibliography}
\end{document}